\documentclass[12pt,reqno] {amsart}%
\usepackage{amsfonts}
\usepackage{amsmath}
\usepackage{amssymb}
\usepackage{xypic,color}
\usepackage{graphicx}%
\usepackage{hyperref}
\usepackage{comment}
 \usepackage{mathrsfs}
\theoremstyle{plain}

\newtheorem{theorem}{Theorem}[section]

\newtheorem{corollary}[theorem]{Corollary}
\newtheorem{definition}[theorem]{Definition}
\newtheorem{lemma}[theorem]{Lemma}
\newtheorem{proposition}[theorem]{Proposition}

\def\bn{\begin{definition}}
\def\en{\end{definition}}
\def\ba{\begin{array}}
\def\ea{\end{array}}
\def\be{\begin{equation}}
\def\ee{\end{equation}}
\def\bd{\begin{description}}
\def\ed{\end{description}}
\def\bu{\begin{enumerate}}
\def\eu{\end{enumerate}}
\def\bi{\begin{itemize}}
\def\ei{\end{itemize}}

\def\<{\langle}
\def\>{\rangle}
\begin{document}

\begin{center}
\textbf{Thompson's semigroup and the first Hochschild cohomology}
\end{center}

\bigskip

\begin{center}
Linzhe Huang\footnote{Yau Mathematical Sciences Center, Tsinghua University, Beijing 100084, China. \\Email: huanglinzhe@mail.tsinghua.edu.cn}
\end{center}

\bigskip
{\bf Abstract.}
In this paper, we apply the theory of algebraic cohomology to study the amenability of Thompson's
group $\mathcal{F}$. We
introduce the notion of unique factorization semigroup which contains Thompson's semigroup $\mathcal{S}$ and the free semigroup $\mathcal{F}_n$ on $n$ generators ($\geq2$). 
 Let $\mathfrak{B}(\mathcal{S})$ and $\mathfrak{B}(\mathcal{F}_n)$ be the Banach algebras generated by the left regular representations of $\mathcal{S}$ and $\mathcal{F}_n$, respectively. It is proved that all derivations on 
 $\mathfrak{B}(\mathcal{S})$ and $\mathfrak{B}(\mathcal{F}_n)$
are automatically continuous, and every derivation on
$\mathfrak{B}(\mathcal{S})$ is
 induced by a bounded linear operator in $\mathcal{L}(\mathcal{S})$, the weak closed Banach algebra consisting of all bounded left convolution operators on
 $l^2(\mathcal{S})$. Moreover, we show that the first continuous
Hochschild cohomology group of $\mathfrak{B}(\mathcal{S})$ with coefficients in $\mathcal{L}(\mathcal{S})$ vanishes.  These conclusions
provide positive indications for the left amenability of Thompson's semigroup.

\bigskip

 
 \bigskip
 
 
\bigskip
\section{Introduction}
 The cohomology theory of associative algebras  was initiated by Hochschild \cite{Ho1,Ho2,Ho3} in 1945 in terms of multilinear maps into a bimodule and coboundary operators. 
 In 1953, after a discussing with Singer 
 at a conference,
Kaplansky went on from there to write his  paper
\cite{Kap} proposing some problems about derivations on $C^*$-algebras and von Neumann algebras. He showed in \cite{Kap} that every derivation 
on a von Neumann algebra $\mathcal{M}$ of type I is inner, which may be restated in cohomological terms  that the first continuous cohomology group of $\mathcal{M}$ vanishes. 

 Recall that a \textbf{derivation} of a Banach algebra $\mathcal{A}$ (over the complex field $\mathbb{C}$) with coefficients in a Banach $\mathcal{A}$-bimodule $X$ is a linear map $D$ from $\mathcal{A}$ into $X$ satisfying $D(AB)=AD(B)+D(A)B$ for all $A$,  $B$ in $\mathcal{A}$. We say that $D$ is \textbf{inner} when there is an element $T$ in $X$ such that $D(A)=AT-TA$ for each $A$ in $\mathcal{A}$. 
Let $\mathcal{Z}^1(\mathcal{A},X)$
be the space of all (continuous) derivations from $\mathcal{A}$ into $X$ and $\mathcal{B}^1(\mathcal{A},X)$ be the space of all inner derivations, respectively. It is clear that $\mathcal{B}^1(\mathcal{A},X)$ is a linear subspace of $\mathcal{Z}^1(\mathcal{A},X)$. The first
(continuous) Hochschild cohomology group  of $\mathcal{A}$ with coefficients in $X$ is then defined to be the following quotient vector space: 
\[H^1(\mathcal{A},X)=\frac{\mathcal{Z}^1(\mathcal{A},X)}{\mathcal{B}^1(\mathcal{A},X)}.\]

The study of cohomology of operator algebras started with the Kadison-Sakai Theorem \cite{Kad, Sa}: Every derivation on a von Neumann algebra $\mathcal{M}$ is inner, i.e., $H^{1}(\mathcal{M},\mathcal{M})=0$. A similar problem is the case when the bimodule is $B(\mathcal{H})$, which is equivalent to Kadison's similarity problem \cite{Kad2}. To study the classification
of von Neumann algebras, from 1968 to 1972,
Johnson, Kadison and Ringrose \cite{KR,KR2,JKR,J}
proved a series of technical
results of the cohomology groups of von Neumann algebras. In particular, they showed that
$H^{n}(\mathcal{M},\mathcal{M})=0$ for all $n\geq1$ when $\mathcal{M}$ is a hyperfinite von Neumann algebra. Due to this fact, Kadison and Ringrose
conjectured that these cohomology groups are zero for
all von Neumann algebras. With the aid of the theory of completely bounded cohomology, this conjecture can be ultimately reduced to the case when $\mathcal{M}$ is a factor of type II$_1$ with  separable predual. In \cite{SS}, Sinclair and Smith showed that
the conjecture holds for von Neumann algebras with 
Cartan subalgebras and separable preduals. Later in 2003, Christensen et al. \cite{CPSS} proved that the continuous cohomology groups $H^n(\mathcal{M},\mathcal{M})$ and $H^n(\mathcal{M}, B(\mathcal{H})) $ of a factor
$\mathcal{M} \subseteq B(\mathcal{H})$ of type II$_1$ with property $\Gamma$ are zero for all $n\geq 1$. The latest result was proved by Pop and Smith in 
\cite{PS}, which shows
that the second cohomology group $H^{2}(\mathcal{M} \overline{\otimes} \mathcal{N},\mathcal{M} \overline{\otimes} \mathcal{N})$ vanishes for arbitrary type II$_1$ von Neumann algebras $\mathcal{M}$ and $\mathcal{N}$. Note that the free group factor $L_{\mathcal{F}_2}$ 
satisfies none
 of the above cases and the higher order
cohomology groups of $L_{\mathcal{F}_2}$  are still unknown.

 The cohomology of Banach algebras are different from
 that of von Neumann algebras in two main aspects:
 the automatically continuity of derivations and the cohomology groups.  It was conjectured 
 by Kaplansky in \cite{Kap} (which was finally proved by Sakai in  \cite{Sa1})
 that every derivation on a C$^*$-algebra is continuous.
 While, derivations on a Banach algebra are not necessarily continuous.  In \cite{BJ}, Bade and Curtis constructed several examples of Banach algebras on which not all derivations are continuous.
 On the other hand, Johnson and Sinclair \cite{JS} showed that the continuity of derivations still holds for semisimple Banach algebras. Up to now, there are no examples showing that the cohomology groups of  a von Neumann algebra  are non-trivial. Let $\mathcal{A}(\mathbb{D})$ be the set of all complex-valued functions that are continuous on the closed unit disk and analytic in the interior.  Then $\mathcal{A}(\mathbb{D})$ endowed with the supremum norm is a
 unital Banach algebra. The second cohomology group of $\mathcal{A}(\mathbb{D})$ is non-trivial  \cite[Proposition 9.1]{Joh}.
 
Thompson's group $\mathcal{F}=\langle X_0,X_1,X_2,\ldots\ |\ X_{i}^{-1}X_{j}X_{i}=X_{j+1},\ j>i\ \rangle$ was firstly introduced by Richard Thompson in 1965 \cite{CFP}.  It was conjectured by  Geoghegan
  around 1979 that: ($\romannumeral1$) the group $\mathcal{F}$ contains no non-abelian free groups; ($\romannumeral2$) $\mathcal{F}$ is not amenable. Statement 
  ($\romannumeral1$) was obtained by Brin and Squier
  \cite{BS}
  in 1985 while ($\romannumeral2$) still remains unknown. Many research works nowadays are
developed to answering this question due to two main reasons: $(1)$ $\mathcal{F}$ is related to many branches of mathematics such as geometric group theory; $(2)$ the amenability problem is one of the most significant research areas in mathematics. Every nontrivial element in $\mathcal{F}$ can be represented as a unique normal form:
\[X_{0}^{\alpha_0}X_{1}^{\alpha_1}\cdots X_{n}^{\alpha_n}X_{n}^{-\beta_n}\cdots X_{0}^{-\beta_0},\]
where $\alpha_{0}$, $\ldots$, $\alpha_{n}$, $\beta_{0}$, $\ldots$, $\beta_{n}$, and $n$ are natural numbers such that ($\romannumeral1$) exactly one of 
$\alpha_n$ and $\beta_n$ is nonzero and ($\romannumeral2$) if $a_k >0$ and $b_k >0$ for some integer $k$ with $0 < k < n$, then $a_{k+1} > 0$ or $b_{k+1} > 0$. For example, $X_0X_1X_{0}^{-1}$
and $X_0X_1X_{2}^{-1}X_{0}^{-1}$ are two normal forms while $X_0X_2X_{3}^{-1}X_{0}^{-1}$
is not.
The amenability of Thompson's group $\mathcal{F}$
has been an open problem for more that 40 years. We refer the readers to \cite{CFP} for more details about it.

Let $G$ be a locally compact group and $l^{\infty}(G)$ be the space of all the bounded complex-valued functions on $G$. Then $l^{\infty}(G)$ is a commutative $C^{*}$-algebra. The group $G$ is said to be \textbf{amenable} if there is a state $\mu$ on $l^{\infty}(G)$ such that $\mu(gf)=\mu(f)$, where $f\in l^{\infty}(G)$, $g\in G$ and $(gf)(h)=f(g^{-1}h)$ for each $h$ in $G$. The state $\mu$ is then called a \textbf{left invariant mean}  on $l^{\infty}(G)$. Additive group of integers $(\mathbb{Z},+)$ is amenable while the free group $\mathcal{F}_2$ on two generators is not amenable. 
In \cite{Joh}, Johnson characterized the amenable group $G$ through the first Hochschild cohomology groups of $l^{1}(G)$,  the space of all absolute-summable complex-valued functions on $G$, with coefficients in the dual Banach $l^{1}(G)$-bimodules:
Let $G$ be a locally compact group. Then $G$ is amenable if and only if  $H^{1}(l^{1}(G),  \mathcal{X}^{*})=0$ for each Banach $l^{1}(G)$-bimodule $\mathcal{X}$.

In this paper, we shall apply the theory of cohomology 
to study the amenability of $\mathcal{F}$. 
Let $\mathcal{S}$ be the subset $\{X_0^{i_0}\cdots X_n^{i_n}\in\mathcal{F}: \ i_j\in\mathbb{N},\ 1\leq j \leq n
\} $ of $\mathcal{F}$. Then $\mathcal{S}=\langle X_0,X_1,X_2,\ldots\ |\ X_{j}X_{i}=X_{i}
X_{j+1},\ j>i\ \rangle$ is a discrete cancellative semigroup and we call it
 \textbf{Thompson's semigroup}. The structure of 
$\mathcal{F}$ is inherited by $\mathcal{S}$ well.
It was proved by Grigorchuk in 1990 that Thompson's group $\mathcal{F}$ is amenable if and only if Thompson's semigroup $\mathcal{S}$
is left amenable.  It is clear that every non-trivial element in $\mathcal{S}$ can be uniquely written as
$X_{i_1}^{\alpha_1}\cdots X_{i_n}^{\alpha_n}$ ($i_1<\cdots <i_n$) for some positive integers $\alpha_j$ ($1\leq j\leq n$). This property is similar to the fundamental theorem of arithmetic which states that every natural number ($\geq 2$) can be uniquely written as the product of primes up to reorder.  Moreover, classical arithmetic functions, 
such as M\"obius function and divisor function,
 can be generalized on $\mathcal{S}$ \cite{Xue}. These properties can help us to understand the structure of $\mathcal{S}$ better.
Therefore,  studying $\mathcal{S}$ may bring more useful tools to study the amenability of $\mathcal{F}$.
Let $\mathfrak{B}(\mathcal{S})$ be the Banach
algebra generated by the left regular representation of
$\mathcal{S}$ and $X$ be a Banach 
$\mathfrak{B}(\mathcal{S})$-bimodule. The Banach algebras $\mathfrak{B}(\mathcal{S})$ and $\mathcal{L}(\mathcal{S})$ (See Section \ref{sec_Preliminaries}) are two important Banach 
$\mathfrak{B}(\mathcal{S})$-bimodules. Essentially speaking, if $H^{n}(\mathfrak{B}(\mathcal{S}),X)\neq0$ for some  $n\geq1$ and bimodule $X$,
then we can almost obtain that $\mathcal{S}$ is not left amenable. The main topic of this paper is to study 
the cohomology of $\mathfrak{B}(\mathcal{S})$.
The basic idea behind the calculation of
cohomology groups is to take an average in a suitable way. In most cases, averages are taken on amenable groups.  In this article, we shall extend it to $\mathcal{S}$. 

The following sections are organized as follows. In Section \ref{sec_Preliminaries}, we provide
some basic definitions related to semigroup algebras and amenable semigroups. In Section \ref{sec_Unique factorization semigroups}, we introduce
the notion of unique factorization semigroup and give three classical examples: Thompson's semigroup, free semigroups and the amenable semigroup $\mathcal{T}$ (See Example \ref{ex_amenable semigroup T}). The continuity of derivations (See Proposition \ref{prop: continuity_derivations})  and the cohomology groups of $\mathfrak{B}(\mathcal{S})$ (See Theorem \ref{thm_the first cohomology is zero} and Theorem
\ref{thm_first_co_B(S)}) are the main contents of this paper and are discussed in Sections   
\ref{sec_Automatically continuous of derivations} and \ref{sec_The first cohomology group}. We end this paper with some further discussions.

\textbf{Acknowledgements.}
The author wishes to thank  Professor Liming Ge for all his long-term help and encouragement.
He would like to thank Boqing Xue and Yongle Jiang for helpful discussions. He also heartly thanks Hanbin Zhang 
for his valuable suggestions on the manuscript.
This research was supported partly by AMSS of  Chinese Academy of Sciences and by  YMSC of Tsinghua
University.

\section{Preliminaries}\label{sec_Preliminaries}
Group algebras and group actions on manifolds are two major sources for the construction of operator
algebras. In applications, generalizations of groups (group algebras), such as semigroups (semigroup algebras), are also used. A brief description of semigroup algebras follows.
 
The Hilbert space $\mathcal{H}$ is $l^{2}(S)$, the space of all square-summable complex valued functions on the semigroup $S$. The semigroup $S$ (with unit $e$) is assumed to be discrete such that $\mathcal{H}$ is separable. The family of functions 
$(\delta_s)_{s\in S}$ forms an orthonormal basis of
$\mathcal{H}$, where $\delta_s$ is $1$ at the semigroup element $s$ and $0$ elsewhere. 
For each $f$ and $g$ in $\mathcal{H}$, let $L_f$ be the left convolution operator defined as 
$L_fg= f\ast g$, where $ f\ast g(s)=\sum_{uv=s}f(u)g(v)$ for each $s$ in $S$. Let $\mathcal{L}(S)$
be the set of all bounded left convolution operators on
$\mathcal{H}$. Then $\mathcal{L}(S)$ is a subalgebra
of $B(\mathcal{H})$. In general, $\mathcal{L}(S)$ is a non-self-adjoint algebra. Similarly, let $\mathcal{R}(S)$ be the subalgebra of $B(\mathcal{H})$ consisting of all bounded right convolution operators. Then
$\mathcal{L}(S)'=\mathcal{R}(S)$ and $\mathcal{R}(S)'=\mathcal{L}(S)$, which implies that  $\mathcal{L}(S)$ and $\mathcal{R}(S)$ are both weak-closed algebras. For each
$s$ in $S$, the operator $L_{\delta_s}$ is an isometry on $\mathcal{H}$ and is denoted by $L_s$ in this paper for convenience. Let $\mathfrak{B}(S)$ be the Banach algebra generated by $\{L_s:\ s\in S \}$ in
$B(\mathcal{H})$. Then
$\mathfrak{B}(S)$ is a Banach subalgebra of $\mathcal{L}(S)$.

Specific examples for such Banach algebras result from
choosing for $S$  any of the free semigroup $\mathcal{F}_n$ on
 $n$ generators ($\geq2$), Thompson's semigroup 
$\mathcal{S}$, or the multiplicative semigroup of natural numbers $(\mathbb{N},\cdot)$.
The algebraic structures of $\mathfrak{B}(S)$ can reflect the structures of $S$ well. In \cite{DHX}, Dong,
Huang and Xue proved that the maximal ideal space of the commutative Banach algebra $\mathfrak{B}(\mathbb{N})$ is homeomorphic to the Cartesian product of unit closed disk indexed by primes (see Theorem 1.1 in \cite{DHX}). They pointed out that this result implies the fundamental theorem of arithmetic. Analogously, to study the cohomology of the 
Banach algebras  $\mathfrak{B}(\mathcal{S})$, $\mathfrak{B}(\mathcal{F}_2)$  
and $\mathfrak{B}(\mathcal{T})$ can help us to understand the properties of the corresponding semigroups well.

In this paper, we shall show that derivations on $\mathfrak{B}(\mathcal{S})$, $\mathfrak{B}(\mathcal{F}_2)$  
and $\mathfrak{B}(\mathcal{T})$  are automatically continuous, and 
every derivation on  $\mathfrak{B}(\mathcal{S})$ is spatial and induced by a bounded operator in $\mathcal{L}(\mathcal{S})$. Comparing with a result of Kadison \cite[Theorem 4]{Kad} that every derivation on a $C^*$-algebra
is spatial, we give a non-trivial example in the case of Banach algebras. Moreover, the first cohomology group
of $\mathfrak{B}(\mathcal{S})$ with coefficients in
$\mathcal{L}(\mathcal{S})$ is shown to be zero, which gives a positive
indication for the left amenability of  Thompson's semigroup. 

In the following, we recall several concepts and results about amenable semigroups. We say that a discrete cancellative semigroup $S$ is \textbf{left (resp. right) amenable}
if there exists a left (resp. right) invariant mean on $l^{\infty}(S)$. For example, the additive semigroup
of natural numbers is amenable while the free semigroup on 
$n$ generators ($\geq2$) is not.
 A \textbf{left (resp. right) F{\o}lner net} of $S$ is a net of non-empty finite subsets $\{F_{\alpha}\}$ in $S$
  such that for any $s\in S$,
  \[\lim_{\alpha}\frac{|sF_{\alpha}\Delta F_{\alpha}|}{|F_{\alpha}|}=0\ \left( resp.\ \lim_{\alpha}\frac{|F_{\alpha}s\Delta F_{\alpha}|}{|F_{\alpha}|}=0\right).\]
It was firstly proved for groups by Følner, and was then generalized to discrete cancellative  semigroups by Frey \cite{Frey} that  $S$ is left (resp. right) amenable if and only if $S$ has
 a left (resp. right ) F{\o}lner's net.
In \cite{Folner} ,  F{\o}lner proved that
every subgroup of an amenable group is still amenable. For semigroups, it is not always true. In
 \cite{Frey}, Frey gave an example of a left amenable semigroup which contains a semigroup 
that is not left amenable. Moreover, he proved the following lemma.
\begin{lemma}\label{thm_ amenability of subsemigroups}
Let $S$ be a cancellative semigroup such that $S$ contains no free subsemigroup on two generators. If $S$ is left
amenable, then every subsemigroup of $S$ is also left amenable.
\end{lemma}
\section{Unique factorization semigroup}\label{sec_Unique factorization semigroups}
\begin{definition}
Let $S$ be a discrete semigroup. We say that $S$
is a \textbf{unique factorization semigroup} if there exists 
a subset $\{X_1,X_2,X_3,\ldots\}$ of $S$ such that

($\romannumeral1$)
every non-trivial element in $S$ can be uniquely written as $X_{i_1}^{\alpha_1}\cdots X_{i_n}^{\alpha_n}$, where $\alpha_i$ ($1\leq i\leq n$) are positive integers and $i_1<\cdots<i_n$; 

($\romannumeral2$) If $e=X_{1}^{\beta_1}\cdots X_{n}^{\beta_n}$ for some non-negative integers $\beta_i$ ($1\leq i\leq n$), then $\beta_i$ must be zero. The subset $\{X_1,X_2,X_3,\ldots\}$ is then called a \textbf{basis}
of $\mathcal{S}$.
\end{definition}

For example, the multiplicative semigroup of natural numbers is a unique factorization semigroup and 
the set of all primes is the unique basis up to reorder.
It is also clear that Thompson's semigroup $\mathcal{S}=\langle X_0, X_1,\cdots\ |\ X_jX_i=X_iX_{j+1},\ i<j \rangle$ is a unique factorization semigroup with the basis $\{X_n\in\mathcal{S}:\ n\in\mathbb{N}\}$. 
Next, we introduce some definitions and properties of $\mathcal{S}$ that will be frequently used in Sections \ref{sec_Automatically continuous of derivations} and \ref{sec_The first cohomology group}.
\begin{definition}\label{def:Index_on_Thompson's_semigroup}
Let $X=X_{0}^{\alpha_0}X_{1}^{\alpha_1}\cdots X_{n}^{\alpha_n}$ be an element in $\mathcal{S}$,
 where $\alpha_i$ ($1\leq i\leq n$) are positive integers and $i_1<\cdots<i_n$. 
We define the \textbf{index of $X$ at the $i^{th}$ position}
 as $\mathsf {ind}_i(X):=\alpha_i$ and the \textbf{index} of $X$ is given by $\mathsf {ind}(X):=\sum_{i=0}^{n}\mathsf {ind}_i(X)$. The \textbf{index} of the unit element $e$ is defined to be zero.
\end{definition}
It is clear that $\mathsf {ind}_{0}(uv)=\mathsf {ind}_{0}(u)+\mathsf {ind}_{0}(v)=\mathsf {ind}_0(vu)$ and $\mathsf {ind}(uv)=\mathsf {ind}(u)+\mathsf {ind}(v)=\mathsf {ind}(vu)$ for all $u$ and $v$
in $\mathcal{S}$. In general, $\mathsf {ind}_i(uv)\neq \mathsf {ind}_i(u)+\mathsf {ind}_i(v)$ for $i\geq2$. For example, $\mathsf {ind}_2(X_0X_2X_1)=\mathsf {ind}_2(X_0X_1X_3)=0\neq 1=
\mathsf {ind}_2(X_0X_2)+\mathsf {ind}_2(X_1)$.

\begin{definition}\cite{Xue}\label{def:Partial_order_ on_Thompson's_ semigroup}
Let $u$, $v$, $w$ $\in$ $\mathcal{S}$.
We call $u$ a \textbf{divisor} of $v$ if $v=uw$ which we denote by $u|v$.
\end{definition}
For example, $X_1|X_0X_2$ while
$X_2\nmid X_0X_2$. It is clear that if $u|v$ then $\mathsf {ind}(u)\leq \mathsf {ind}(v)$.
\begin{lemma}\cite{Xue}\label{lem:Partial_order_ on_Thompson's_ semigroup}
The relation ``$|$" is a \textbf{partial order} on $\mathcal{S}$. 
\begin{proof}
Let $u$, $v$, $w$ $\in$ $\mathcal{S}$. Next, we verify the following three axioms of the partial order.
\begin{enumerate}
\item (Reflexivity.) $u=ue$ implies $u|u$.
\item (Antisymmetry.) If $u|v$ and $v|u$, then $u=vw_1$ and $v=uw_2$ for some $w_1$, $w_2$
in $\mathcal{S}$. Then we have $w_2w_1=e$,
which imples $w_1=w_2=e$. Thus $u=v$.
\item  (Transitivity.) Suppose that $u|v$ and $v|w$, 
then $v=uw_1$ and $w=vw_2$ for some $w_1$, $w_2$ in $\mathcal{S}$. We have $w=uw_1w_2$.
Thus $u|w$.
\end{enumerate} 
As a result, we conclude that ``$|$" is a partial order.
\end{proof}
\end{lemma}

\begin{lemma}\label{lem:partial_order_S_property}
Let $X$ be an element in $\mathcal{S}$ such that $X_0|X$. Then
for each  $n\in\mathbb{N}$, we have

 ($\romannumeral1$) $X_{1}^{-n}XX_{1}^{n}\in \mathcal{S}$ if and only if
 $X_{1}^{n}|X$;
 
  ($\romannumeral2$) $X_{1}^{n}XX_{1}^{-n}\in 
\mathcal{S} $ if and only if $X=YX_{1}^{n}$ for some $Y\in \mathcal{S}$.
\begin{proof}
 ($\romannumeral1$) If $X_{1}^{n}|X$, then it is clear that $X_{1}^{-n}XX_{1}^{n}\in \mathcal{S}$. Conversely, we have $X_1^n|XX_1^n$.
 If $X_1^n$ is not a divisor of $X$, since
 $X_1^{-n}X_0=X_0X_2^{-n}$, we have that the
 normal form of $X_{1}^{-n}XX_{1}^{n}$ is $ZX_j^{-m}$ for some $Z\in\mathcal{S}$ and $j\geq2$, $m\geq1$. This leads to a contradiction. Thus 
 $X_1^n|X$.
 ($\romannumeral2$) Assume that $W=X_{1}^{n}XX_{1}^{-n}\in \mathcal{S} $, then we have
 $X_0|W$ and $X=X_1^{-n}WX_1^{n}\in \mathcal{S}$. From ($\romannumeral1$), we have $X_1^n|W$. Let $Y= X_1^{-n}W$, then $X=YX_1^n$. The other direction is obvious. We complete the proof.
\end{proof}
\end{lemma}

 Finally, with the aid of the index we introduce a total order on $\mathcal{S}$ which plays an important role in Sections \ref{sec_Automatically continuous of derivations} and \ref{sec_The first cohomology group}.
 
\begin{definition}\label{def:Total_order_on_ Thompson's_semigroup}
 Let $u$, $v$ $\in$ 
 $\mathcal{S}$. We write  $u\prec v$ if
one of the following conditions holds: 

($\romannumeral1$)
$\mathsf {ind}(u)<\mathsf {ind}(v)$; 

($\romannumeral2$)
$\mathsf {ind}(u)=\mathsf {ind}(v)$ and  $\mathsf {ind}_{0}(u)>\mathsf {ind}_{0}(v)$;

($\romannumeral3$) $\mathsf {ind}(u)=\mathsf {ind}(v)$, $\mathsf {ind}_{0}(u)=\mathsf {ind}_{0}(v)$, and there exits a positive integer $i$ such that
 $\mathsf {ind}_{j}(u)=\mathsf {ind}_{j}(v)$ whenever $j<i$ while $\mathsf {ind}_{i}(u)>\mathsf {ind}_{i}(v)$. 
 
We use $u\preceq v$ to denote $u\prec v$ or $u=v$. 
 \end{definition}
 For example, $X_0\prec X_1\prec X_0X_1\prec X_0X_2
 \prec X_1X_2$. The relation ``$\preceq$"
 is a \textbf{total order} on $\mathcal{S}$
with the following properties. 
\begin{lemma}\label{lem_property of total order of Thompson's semigroup} 
We have the following:

 ($\romannumeral1$) There exists a unique minimal element in every non-empty subset  of $\mathcal{S}$ under the total order;
 
  ($\romannumeral2$) 
Let $u_i$  and $v_i$ ($1\leq i\leq n$) be $2n$
elements in $\mathcal{S}$.
If $u_{i}\preceq v_{i}$ for each $1\leq i\leq n$, then $\Pi_{i=1}^{n}u_{i}\preceq \Pi_{i=1}^{n}v_{i}$. The equality holds if and only if $u_{i}=v_{i}$ 
 for each $1\leq i\leq n$.
\begin{proof}
It is clear that  ($\romannumeral1$) holds.
We  now give the proof of  ($\romannumeral2$).
Firstly, we consider the case when $n=2$.
If $\mathsf {ind}(u_{1})<\mathsf {ind}(v_{1})$ or $\mathsf {ind}(u_{2})<\mathsf {ind}(v_{2})$, then
$\mathsf {ind}(u_{1}u_{2})=\mathsf {ind}(u_{1})+\mathsf {ind}(u_{2})<\mathsf {ind}(v_{1})+\mathsf {ind}(v_{2})=\mathsf {ind}(v_{1}v_{2})$. This implies $u_{1}u_{2}\prec v_{1}v_{2}$. In the case when $\mathsf {ind}(u_{1})=\mathsf {ind}(v_{1})$ and $\mathsf {ind}(u_{2})=\mathsf {ind}(v_{2})$, if $\mathsf {ind}_{0}(u_{1})>\mathsf {ind}_{0}(v_{1})$ or $\mathsf {ind}_{0}(u_{2})>\mathsf {ind}_{0}(v_{2})$, then
$\mathsf {ind}_{0}(u_{1}u_{2})=\mathsf {ind}_{0}(u_{1})+\mathsf {ind}_{0}(u_{2})>\mathsf {ind}_{0}(v_{1})+\mathsf {ind}_{0}(v_{2})=\mathsf {ind}_{0}(v_{1}v_{2})$. This also implies $u_{1}u_{2}\prec v_{1}v_{2}$. Hence, we may assume that $\mathsf {ind}_{0}(u_{1})=\mathsf {ind}_{0}(v_{1})$ and $\mathsf {ind}_{0}(u_{2})=\mathsf {ind}_{0}(v_{2})$. If either $\mathsf {ind}(u_{1})$ or $\mathsf {ind}(u_{2})$ is zero, then it is trivial. Thus we may further assume that $\mathsf {ind}(u_{1})=\mathsf {ind}(v_{1})\geq1$ and $\mathsf {ind}(u_{2})=\mathsf {ind}(v_{2})\geq1$.

\textbf{Case I:} $u_{2}=v_{2}$. If $u_{1}=v_{1}$, then it is proved. Otherwise, let $u_{1}=X_{0}^{\alpha_{0}}\cdots X_{n}^{\alpha_{n}}$ and $v_{1}=X_{0}^{\beta_{0}}\cdots X_{m}^{\beta_{m}}$. 
By the definition of the total order, there exists an $i>0$ such that
$\alpha_{j}=\beta_{j}$ whenever $j<i$ and
$\alpha_{i}>\beta_{i}$. Since $\{X_l\in\mathcal{S}:\ l\in\mathbb{N}\}$ is a basis of
 $\mathcal{S}$,
it only needs to prove 
$u_{1}X_{l}\prec v_{1}X_{l}$ for each 
$X_{l}\in \mathcal{S}$. In fact, we have
\begin{align} \label{eq1 of lem property of total order of Thompson's semigroup }
u_{1}X_{l}=
\begin{cases}
X_{0}^{\alpha_{0}}\cdots X_{l-1}^{\alpha_{l-1}}X_{l}^{\alpha_{l}+1} X_{l+2}^{\alpha_{l+1}}\cdots X_{i+1}^{\alpha_{i}}
\cdots X_{n+1}^{\alpha_{n}}&l<i,
\cr X_{0}^{\alpha_{0}}\cdots X_{i-1}^{\alpha_{i-1}}X_{i}^{\alpha_{i}+1}  X_{i+2}^{\alpha_{i+1}}\cdots X_{n+1}^{\alpha_{n}}&l=i,
\cr X_{0}^{\alpha_{0}}\cdots X_{i-1}^{\alpha_{i-1}}X_{i}^{\alpha_{i}}  X_{i_{1}}^{\alpha'_{i_{1}}}\cdots X_{i_{t}}^{\alpha'_{i_{t}}}&l>i,
\end{cases}
\end{align}
and
\begin{align} \label{eq2 of lem property of total order of Thompson's semigroup }
v_{1}X_{l}=
\begin{cases}
X_{0}^{\beta_{0}}\cdots X_{l-1}^{\beta_{l-1}}X_{l}^{\beta_{l}+1} X_{l+2}^{\beta_{l+1}}\cdots X_{i+1}^{\beta_{i}}
\cdots X_{m+1}^{\beta_{m}}&l<i,
\cr X_{0}^{\beta_{0}}\cdots X_{i-1}^{\beta_{i-1}}X_{i}^{\beta_{i}+1}  X_{i+2}^{\beta_{i+1}}\cdots X_{m+1}^{\beta_{m}}&l=i,
\cr X_{0}^{\beta_{0}}\cdots X_{i-1}^{\beta_{i-1}}X_{i}^{\beta_{i}}  X_{j_{1}}^{\beta'_{j_{1}}}\cdots X_{j_{s}}^{\beta'_{j_{s}}}&l>i,
\end{cases}
\end{align}
where $i<i_{1}<\cdots<i_{t}$ and $i<j_{1}<\cdots<j_{s}$. By comparing equations (\ref{eq1 of lem property of total order of Thompson's semigroup }) with (\ref{eq2 of lem property of total order of Thompson's semigroup }), we have $u_{1}X_{l}\prec u_{2}X_{l}$.

\textbf{Case II:} $u_{2}\prec v_{2}$. Let $u_{2}=X_{0}^{\gamma_{0}}\cdots X_{k}^{\gamma_{k}}$ and $v_{2}=X_{0}^{\omega_{0}}\cdots X_{l}^{\omega_{l}}$, then
there exists an $h>0$ such that $\gamma_{j}=\omega_{j}$ whenever $j<h$ and
$\gamma_{h}>\omega_{h}$.
By the first case, it can be reduced to the case when $u_{2}=X_{h}^{\gamma_{h}}\cdots X_{k}^{\gamma_{k}}$ and $v_{2}=X_{h}^{\omega_{h}}\cdots X_{l}^{\omega_{l}}$. If $u_{1}=v_{1}=X_{0}^{\alpha_{0}}\cdots X_{n}^{\alpha_{n}}$, then
\begin{align} \label{eq3 of lem property of total order of Thompson's semigroup }
u_{1}u_{2}=
\begin{cases}
X_{0}^{\alpha_{0}}\cdots X_{n}^{\alpha_{n}}X_{h}^{\gamma_{h}}\cdots X_{l}^{\gamma_{l}}&h>n,
\cr X_{0}^{\alpha_{0}}\cdots X_{h-1}^{\alpha_{h-1}}X_{h}^{\alpha_{h}+\gamma_{h}}  X_{h_{1}}^{\alpha'_{h_{1}}}\cdots
X_{h_{t}}^{\alpha'_{h_{t}}}&h\leq n,
\end{cases}
\end{align}
and
\begin{align} \label{eq4 of lem property of total order of Thompson's semigroup }
v_{1}v_{2}=
\begin{cases}
X_{0}^{\alpha_{0}}\cdots X_{n}^{\alpha_{n}}X_{h}^{\omega_{h}}\cdots X_{l}^{\omega_{l}}&h>n,
\cr X_{0}^{\alpha_{0}}\cdots X_{h-1}^{\alpha_{h-1}}X_{h}^{\alpha_{h}+\omega_{h}}  X_{h_{1}}^{\alpha''_{h_{1}}}\cdots
X_{h_{s}}^{\alpha''_{h_{s}}}&h\leq n,
\end{cases}
\end{align}
where $h<h_{1}<\cdots<h_{max\{t,s\}}$. By comparing equations (\ref{eq3 of lem property of total order of Thompson's semigroup }) with (\ref{eq4 of lem property of total order of Thompson's semigroup }), we have $u_{1}u_{2}\prec v_{1}v_{2}$. The proof of the case when
$u_{1}\prec v_{1}$ is similar, we omit it here. Moreover, we can obtain from the above process directly that $u_1u_2=v_1v_2$ if and only if $u_1=v_1$ and $u_2=v_2$.

The general case when $n>2$ can be obtained by induction.
\end{proof}
\end{lemma}

We now turn to the free semigroup $\mathcal{F}_n$ on $n$ ($\geq 2$) generators $a_i$ ($1\leq i \leq n$).
The following two definitions on $\mathcal{F}_n$
are parallel to Definitions \ref{def:Index_on_Thompson's_semigroup} and \ref{def:Total_order_on_ Thompson's_semigroup}, and Lemma \ref{lem:Fn_total_order_property} is parallel to
Lemma \ref{lem_property of total order of Thompson's semigroup}.
\begin{definition}
\label{def:Index_on_the_free_semigroup} Let $g=\prod_{j=1}^m \prod_{k=1}^{n}a_k^{i_{jk}}$ be an element in $\mathcal{F}_n$, where $i_{jk}$
 ($1\leq j \leq m$, $1\leq k\leq n$) are non-negative  integers. The \textbf{index} of $g$ is defined as $ \mathsf {ind}(g):= \sum_{j=1}^m \sum_{k=1}^{n}i_{jk}$. 
 \end{definition}
 It is clear that 
 $\mathsf {ind}(gh)=\mathsf {ind}(g)+\mathsf {ind}(h)$ for all $g$ and $h$ in $\mathcal{F}_n$.
 \begin{definition}
\label{def:Total_order_on_the_free_semigroup} Let
 $g=\prod_{j=1}^m \prod_{k=1}^{n}a_k^{i_{jk}}$
 and  $h=\prod_{j=1}^m \prod_{k=1}^{n}a_k^{i'_{jk}}$ be two  elements in $\mathcal{F}_n$.
We write $g\prec h$ if one of the following conditions holds:

 ($\romannumeral1$) $\mathsf {ind}(g)<\mathsf {ind}(h)$; 
 
 ($\romannumeral2$) $\mathsf {ind}(g)=\mathsf {ind}(h)$ and there exist 
some $j_0$ ($1\leq j_0\leq m$) and  $k_0$ ($1\leq k_0\leq n$) such that $i_{j_0k_0}> i'_{j_0k_0}$ and $i_{jk}= i'_{jk}$ whenever $j< j_0$ or  $j= j_0$ and $k<k_0$. 

We use $g\preceq h$ to denote $g\prec h$ or $g=h$.
\end{definition}
 The relation ``$\preceq$"  is a \textbf{total order} on $\mathcal{F}_n$. 
\begin{lemma} \label{lem:Fn_total_order_property}
We have the following:

  ($\romannumeral1$) There exists a unique minimal element in each subset  of $\mathcal{F}_n$ under the total order;
  
  ($\romannumeral2$) 
Let $g_i$  and $h_i$ ($1\leq i\leq m$) be $2m$
elements in $\mathcal{F}_n$.
If $g_{i}\preceq h_{i}$ for each $i$ ($1\leq i\leq m$), then $\Pi_{i=1}^{m}g_{i}\preceq \Pi_{i=1}^{m}h_{i}$. Moreover, the equality holds if and only if $g_{i}=h_{i}$ 
 for each $i$ ($1\leq i\leq m$).
\end{lemma}

\begin{proposition}\label{ex_free semigroups are unique factorization semigroups}
For $n\geq 2$,
free semigroup $\mathcal{F}_n$ on $n$ generators $a_i$ ($1\leq i \leq n$) is a unique factorization semigroup \footnote{We thank D. Wu for 
his useful suggestion on the proof of Example 
\ref{ex_free semigroups are unique factorization semigroups}.}.
\begin{proof}
Let $X_i=a_i$ ($1\leq i\leq n$), then $X_1\prec X_2
\prec \cdots \prec X_n$. Let $X_{n+1}$ be the minimal element of $\mathcal{F}$ such that $X_{n+1}$ cannot be represented as $X_1^{i_1}\cdots 
X_n^{i_n}$ for any non-negative integers $i_j$ ($1\leq j\leq n$). Then $X_n\prec X_{n+1}$.
 Let $X_{n+2}$ be the minimal element of $\mathcal{F}$ such that $X_{n+2}$ cannot be represented as $X_1^{i_1}\cdots 
X_n^{i_{n+1}}$ for any non-negative integers $i_j$ ($1\leq j\leq n+1$). Then $X_{n+1}\prec X_{n+2}$.
Continuing this process, we can obtain a subset $\{X_1,X_2,\ldots\}$ of $\mathcal{F}_n$ such that
$X_i\prec X_{i+1}$ for each $i$ ($\geq 1$). Then every element can be written as the product 
$X_1^{i_1}\cdots X_n^{i_n}$ for some non-negative integers $i_j$ ($1\leq j\leq n$). We shall show the uniqueness by induction on the index. It is clear that the uniqueness holds for index $\leq1$.
Assume that the uniqueness holds for index $\leq k$ 
($k\geq 1$).
Let $g\in \mathcal{F}_n$ and $\mathsf{ind}(g)= k+1$.
Suppose that $g$ has two different forms:
\[g= X_{i_1}^{\alpha_1}\cdots X_{i_n}^{\alpha_n}= X_{j_1}^{\beta_1}\cdots X_{j_m}^{\beta_m},\]
where $\alpha_i$ ($1\leq i\leq n$) and $\beta_j$ ($1\leq j\leq m$) are positive integers, $i_1<\cdots<i_n$ and $j_1<\cdots<j_m$. If $i_1=j_1$, then     
\[X_{i_1}^{\alpha_1-1}\cdots X_{i_n}^{\alpha_n}= X_{j_1}^{\beta_1-1}\cdots X_{j_m}^{\beta_m},\]
which implies that $n=m$,  
$i_t=j_t$ and $\alpha_t=\beta_t$ ($1\leq t\leq n$). 
This leads to a contradiction.

 We may assume that 
$i_1<j_1$. Then there exists some $l$ ($1\leq l\leq n-1$) such that
$(X_{i_1}^{\alpha_1}\cdots X_{i_l}^{\alpha_l})^{-1}X_{j_1}$
belongs to $\mathcal{F}_n$ and $1\leq \mathsf{ind}((X_{i_1}^{\alpha_1}\cdots X_{i_l}^{\alpha_l})^{-1}X_{j_1})< \mathsf{ind}(X_{i_{l+1}})$,  or $(X_{i_1}^{\alpha_1}\cdots X_{i_{l-1}}^{\alpha_{l-1}}$
$X_{i_l}^{\alpha'_l})^{-1}X_{j_1}$ ($1\leq l\leq n$)
belongs to $\mathcal{F}_n$ for some $\alpha'_l$ ($1\leq\alpha'_l< \alpha_l$) and $1\leq
\mathsf{ind}((X_{i_1}^{\alpha_1}\cdots X_{i_{l-1}}^{\alpha_{l-1}}X_{i_l}^{\alpha'_l})^{-1}$ $X_{j_1})< \mathsf{ind}(X_{i_l})$. In the first case, we have
$(X_{i_1}^{\alpha_1}\cdots X_{i_l}^{\alpha_l})^{-1}X_{j_1}=X_{s_1}^{\gamma_1}\cdots X_{s_t}^{\gamma_t}$ for some positive integers $\gamma_1,\ldots,\gamma_t$ and 
$s_1<\cdots<s_t<min\{ i_{l+1},j_1\}$. Then
\[ X_{i_{l+1}}^{\alpha_{l+1}}\cdots X_{i_n}^{\alpha_n}= X_{s_1}^{\gamma_1}\cdots X_{s_t}^{\gamma_t}X_{j_1}^{\beta_1-1}\cdots X_{j_m}^{\beta_m},\]
which leads to a contradiction. In the second case, we have
$(X_{i_1}^{\alpha_1}\cdots X_{i_l}^{\alpha'_l})^{-1}X_{j_1}=X_{s_1}^{\gamma_1}\cdots X_{s_t}^{\gamma_t}$ for some positive integers $\gamma_1,\ldots,\gamma_t$ and 
$s_1<\cdots<s_t<min\{ i_{l},j_1\}$. Then
\[ X_{i_l}^{\alpha_i-\alpha'_i}X_{i_{l+1}}^{\alpha_{l+1}}\cdots X_{i_n}^{\alpha_n}= X_{s_1}^{\gamma_1}\cdots X_{s_t}^{\gamma_t}X_{j_1}^{\beta_1-1}\cdots X_{j_m}^{\beta_m},\]
which also leads to a contradiction. Above all, we complete the proof. 
\end{proof}
\end{proposition}

It is known that Thompson's semigroup is not right amenable while the 
left amenability is still unknown. The free semigroup 
$\mathcal{F}_n$ is neither left nor right amenable.
For completeness, we construct the following example.
\begin{proposition}\label{ex_amenable semigroup T}
Let $\mathcal{T}$ be the semigroup generated by 
\begin{equation*}
A=
\begin{pmatrix}
1 & 0 \\
0 & 2
\end{pmatrix},\
B=
\begin{pmatrix}
2 & 0 \\
0 & 1
\end{pmatrix},\
C=
\begin{pmatrix}
0 & 2 \\
3 & 0
\end{pmatrix}
\end{equation*}
in $GL_{2}(\mathbb{Z})$.  Then $\mathcal{T}$ is a
non-commutative unique factorization semigroup. Moreover, $\mathcal{T}$ is both left and right amenable.
\begin{proof}
 It can be verified directly that
$AB=BA$, $AC=CB$, $BC=CA$. Hence $\mathcal{T}$
is non-commutative and every element in
$\mathcal{T}$ can be written as $A^{\alpha_1}B^{\alpha_2}C^{\alpha_3}$
for some natural numbers $\alpha_1$, $\alpha_2$ and $\alpha_3$.
Let $X$ be a matrix in $\mathcal{T}$ with two different representations: $X=A^{\alpha_{1}}B^{\alpha_{2}}C^{\alpha_{3}}=A^{\beta_{1}}B^{\beta_{2}}C^{\beta_{3}}$, 
where $\alpha_i$ and $\beta_i$ ($1\leq i\leq 3$) are natural numbers.  Take the determinant at the both side, we have $2^{\alpha_{1}}\cdot 2^{\alpha_{2}}\cdot(-6)^{\alpha_{3}}=2^{\beta_{1}}\cdot 2^{\beta_{2}}\cdot(-6)^{\beta_{3}}$.
 Hence $\alpha_{3}=\beta_{3}$ and $\alpha_{1}+\alpha_{2}=\beta_{1}+\beta_{2}$. Then we have $A^{\alpha_{1}-\beta_{1}}=B^{\beta_{2}-\alpha_{2}}$, which implies $\alpha_{1}=\beta_{1}$ and $\alpha_{2}=\beta_{2}$. Thus $\mathcal{T}$ is a unique factorization semigroup.
   For each positive integer $N$ ($\geq2$), let $F_{N}=\{A^{\alpha_{1}}B^{\alpha_{2}}C^{\alpha_{3}}\in\mathcal{T}\ |\ 0\leq \alpha_{i}\leq N-1,\ i=1,2,3\}$, we have
\begin{align*}
AF_{N}\cap F_{N}=\{A^{\alpha_{1}}B^{\alpha_{2}}C^{\alpha_{3}}\in \mathcal{T}\ |\ 1\leq \alpha_{1}\leq N-1,\ 0\leq \alpha_{i}\leq N-1,\ i=2,3\},\\
BF_{N}\cap F_{N}=\{A^{\alpha_{1}}B^{\alpha_{2}}C^{\alpha_{3}}\in \mathcal{T}\ |\ 1\leq \alpha_{2}\leq N-1,\ 0\leq \alpha_{i}\leq N-1,\ i=1,3\},\\
CF_{N}\cap F_{N}=\{A^{\alpha_{1}}B^{\alpha_{2}}C^{\alpha_{3}}\in \mathcal{T}\ |\ 1\leq \alpha_{3}\leq N-1,\ 0\leq \alpha_{i}\leq N-1,\ i=1,2\}
\end{align*}
and $|AF_{N}\cap F_{N}|=|BF_{N}\cap F_{N}|=|CF_{N}\cap F_{N}|=N^{2}(N-1)$. Then
\[\lim_{N\rightarrow\infty}\frac{|AF_{N}\cap F_{N}|}{|F_{N}|}=\lim_{N\rightarrow\infty}\frac{|BF_{N}\cap F_{N}|}{|F_{N}|}=\lim_{N\rightarrow\infty}\frac{|CF_{N}\cap F_{N}|}{|F_{N}|}=\lim_{N\rightarrow\infty}\frac{N^{2}(N-1)}{N^{3}}=1.\]
Since $A$, $B$, and $C$ are generators of $\mathcal{T}$, then $(F_N)_{N\in \mathbb{N}}$ is a left F{\o}lner
sequence of $\mathcal{T}$.
 Thus $\mathcal{T}$ is left amenable.
 On the other hand, we have
\begin{align*}
F_{N}A\cap F_{N}=&\{A^{\alpha_{1}}B^{\alpha_{2}}C^{\alpha_{3}}\in \mathcal{T}\ |\ 1\leq \alpha_{1}\leq N-1,\ 0\leq \alpha_{i}\leq N-1,\ i=2,3,\ \alpha_{3}\ \text{is even}\}\\
\cup&\{A^{\alpha_{1}}B^{\alpha_{2}}C^{\alpha_{3}}\in \mathcal{T}\ |\ 1\leq \alpha_{2}\leq N-1,\ 0\leq \alpha_{i}\leq N-1,\ i=1,3,\ \alpha_{3}\ \text{is odd}\},\\
F_{N}B\cap F_{N}=&\{A^{\alpha_{1}}B^{\alpha_{2}}C^{\alpha_{3}}\in \mathcal{T}\ |\ 1\leq \alpha_{2}\leq N-1,\ 0\leq \alpha_{i}\leq N-1,\ i=1,3,\ \alpha_{3}\ \text{is even}\}\\
\cup&\{A^{\alpha_{1}}B^{\alpha_{2}}C^{\alpha_{3}}\in \mathcal{T}\ |\ 1\leq \alpha_{1}\leq N-1,\ 0\leq \alpha_{i}\leq N-1,\ i=2,3,\ \alpha_{3}\ \text{is odd}\},\\
F_{N}C\cap F_{N}=&\{A^{\alpha_{1}}B^{\alpha_{2}}C^{\alpha_{3}}\in \mathcal{T}\ |\ 1\leq \alpha_{3}\leq N-1,\ 0\leq \alpha_{i}\leq N-1,\ i=1,2\}.
\end{align*}
Similarly, $(F_N)_{N\in \mathbb{N}}$ is also a right 
F{\o}lner sequence of $\mathcal{T}$. Thus $\mathcal{T}$ is both left and right amenable. We complete the proof.
\end{proof}
 \end{proposition}
\section{The continuity of derivations}\label{sec_Automatically continuous of derivations}
In this section, we shall prove the following theorem.

\begin{theorem}\label{prop: continuity_derivations}
Derivations on the Banach algebras $\mathfrak{B}(\mathcal{S})$ and $\mathfrak{B}(F_n)$ are  continuous.
\end{theorem}
For the sake of proof, we firstly introduce the following definition.
\begin{definition}
A discrete semigroup $S$ is said to be \textbf{lower stable} if there is a total order
$``\preceq"$ on $S$ such that 

($\romannumeral1$) There exists a unique minimal element in each subset  of $S$ under the total order;

  ($\romannumeral2$) 
Let $u_i$  and $v_i$ ($1\leq i\leq n$) be $2n$
elements in $S$.
If $u_{i}\preceq v_{i}$ for each $1\leq i\leq n$, then $\Pi_{i=1}^{n}u_{i}\preceq \Pi_{i=1}^{n}v_{i}$. Moreover, the equality holds if and only if $u_{i}=v_{i}$  for each $1\leq i\leq n$.
\end{definition}
By Lemmas \ref{lem_property of total order of Thompson's semigroup} and \ref{lem:Fn_total_order_property}, we have that
 Thompson's semigroup $\mathcal{S}$ and the free semigroup $\mathcal{F}_n$ ($n\geq2$) are both lower stable.
 
\begin{definition} 
A Banach algebra $\mathcal{A}$  is said to be \textbf{semisimple} if its Jacobson radical $\mathcal{J}$ equals zero, where
\[\mathcal{J}=\bigcap_{maximal\ ideals\ of\
\mathcal{A}}\mathcal{I}=\bigcap_{maximal\ left\ ideals\ of\ \mathcal{A}}\mathcal{I}_{l}=
\bigcap_{maximal\ right\ ideals\ of\ \mathcal{A}}\mathcal{I}_{r}.\]
\end{definition}
The following lemma is a characterization of 
semisimple Banach algebras.
\begin{lemma}\label{lem_a characterization of semisimple}
Let $\mathcal{A}$ be a Banach algebra with the unit
$I$. If $\mathcal{A}$ contains no non-zero quasi-nilpotent
operators, then $\mathcal{A}$ is semisimple.
\begin{proof}
Let $T$ be a non-zero element in $\mathcal{A}$.
Let $\lambda$ be a non-zero spectrum  of $T$, then  
at least one of
$(\lambda I-T)\mathcal{A}$ and $\mathcal{A}(\lambda I-T)$ is properly contained in
 $\mathcal{A}$.
We may assume that the right ideal $(\lambda I-T)\mathcal{A}$ is properly contained in $\mathcal{A}$, then there exists a maximal right ideal  $I_{r}$ 
of $\mathcal{A}$ such that
$(\lambda I-T)\mathcal{A}\subseteq I_{r}\subset\mathcal{A}$. Thus $T$ does not belong to $I_{r}$. This implies $T$ is not in the Jacobson 
radical $\mathcal{J}$ of
$\mathcal{A}$. Consequently, we have $\mathcal{J}=0$. This completes the proof.
\end{proof}
\end{lemma}

\begin{lemma}\label{prop:lower_stable_semisimple}
Let $S$ be a lower stable discrete semigroup and $\mathfrak{B}(S)$ be the Banach algebra generated by
$\{L_s:\ s\in S\}$ in $B(l^2(S))$. Then $\mathfrak{B}(S)$ is semisimple.
\begin{proof}
Let $L_f$ be a non-zero element in $\mathfrak{B}(S)$. We claim that the spectral radius  $r(L_f)>0$. 
By the definition of lower stable semigroup, there is a unique minimal element $X$ of the subset
$\{X\in S:f(X)\neq0\}$ under the total order. Then we have 
\[\underbrace{f\ast\cdots \ast f}_{n}(X^{n})=f(X)^{n}\] for each $n\geq1$. Therefore,
\begin{align*}
r(L_{f})=lim_{n\rightarrow\infty}\|L_{f}^{n}\|^{1/n}&\geq lim_{n\rightarrow\infty}\|\underbrace{f\ast\cdots \ast f}_{n}\|_{2}^{1/n}\\
&\geq lim_{n\rightarrow\infty}|\underbrace{f\ast\cdots \ast f}_{n}(X^{n})|^{1/n}=|f(X)|>0.
\end{align*}

By Lemma \ref{lem_a characterization of semisimple}, we obtain that $\mathfrak{B}(S)$ is semisimple.
\end{proof}
\end{lemma}
\begin{corollary}\label{cor:S_Fn_semisimple}
The Banach algebras $\mathfrak{B}(\mathcal{S})$
and $\mathfrak{B}(\mathcal{F}_n)$ are semisimple.
\end{corollary}
The semigroup $\mathcal{T}$ in Example
\ref{ex_amenable semigroup T} may not be lower stable. However, the Banach algebra  $\mathfrak{B}(\mathcal{T})$ is
 semisimple. The proof is similar with that of Proposition
 \ref{prop:lower_stable_semisimple}, we omit it here.
 Now, we are ready to prove Proposition \ref{prop: continuity_derivations}.
\begin{proof}[Proof of Proposition \ref{prop: continuity_derivations}]
Theorem 4.1 of \cite{JS} states that derivations on a semisimple Banach algebra are continuous, then by Corollary \ref{cor:S_Fn_semisimple} 
we can obtain the conclusion.
\end{proof}
\section{The first cohomology group of $\mathfrak{B}(\mathcal{S})$}\label{sec_The first cohomology group}
A derivation $D$ on a Banach algebra $\mathfrak{B}\subseteq B(\mathcal{H})$ is said to be \textbf{spatial} if there exists a bounded operator $T$ in $B(\mathcal{H})$ such that 
$D(A)=TA-AT$ for each $A$ in $\mathfrak{B}$.
 In \cite{Kad}, Kadison proved that all derivations on a $C^{*}$-algebra are spatial. In this section, we will prove that all derivations on the Banach algebra $\mathfrak{B}(\mathcal{S})$ are spatial and are induced by bounded operators in $\mathcal{L}(\mathcal{S})$ (Corollary \ref{cor:derivation_spatial}). Moreover, the first continuous cohomology group of $\mathfrak{B}(\mathcal{S})$
with coefficients in $\mathcal{L}(\mathcal{S})$ is
shown to be zero (Theorem \ref{thm_the first cohomology is zero}).

The Hilbert space $\mathcal{H}$ is $l^2(\mathcal{S})$. Recall that $\mathfrak{B}(\mathcal{S})$ is the Banach algebra generated by $\{L_s:s\in \mathcal{S} \}$ in $B(\mathcal{H})$ and $\mathcal{L}(\mathcal{S})=\left\{L_{f}\in B(\mathcal{H}):f\in \mathcal{H}\right\}$. We use  $\sum_{X\in \mathcal{S}}f(X)X$ and $\sum_{X\in \mathcal{S}}\overline{f(X)}X^{*}$ to denote $L_{f}$ and its adjoint operator for convenience. It is clear that
$XX^{*}$ is the projection from $\mathcal{H}$ onto the closure of the
subspace span$\{\delta_Y:\ Y\in \mathcal{S},\ X|Y\}$ and $X^{*}X=I$ (the identity operator). Recall that  ``|" is the partial order of Thompson's semigroup  introduced in  Section \ref{sec_Unique factorization semigroups} and  ``$X|Y$" means that there exists a $Z$ in $\mathcal{S}$ such that $Y=XZ$.

Let $(\mathbb{N},+)$ be the additive semigroup of natural numbers. 
The notation $\beta \mathbb{N}$ denotes the maximal ideal space of $l^{\infty}(\mathbb{N})$, and
the elements in $\beta \mathbb{N} \setminus\mathbb{N}$ are called \textbf{free ultrafilters}. 
Let $\omega\in \beta\mathbb{N} \setminus \mathbb{N}$ be a given free ultrafilter. For any $n\in\mathbb{N} $ and any $f$ in
$l^{\infty}(\mathbb{N})$, we define $E_n(f) =\frac{1}{n} \sum_{i=0}^{n-1} f(i)$. Then, for each given $f$, the function $n\rightarrow E_n(f)$ gives
rise to another function in $l^{\infty}(\mathbb{N})$.
By Gelfand-Naimark theorem, we have $l^{\infty}(\mathbb{N})\cong C(\beta\mathbb{N})$. Thus $E_n(f)$ is a continuous function on $\beta\mathbb{N}$.
 The limit of $E_n(f)$ at $\omega$ is denoted by $E_{\omega}(f)$ or the integral $\int_\mathbb{N} f(n) dE_{\omega}(n)$. Then $E_{\omega}$ is an
invariant mean defined on $l^{\infty}(\mathbb{N})$,
 i.e., $\int_\mathbb{N} f(n) dE_{\omega}(n)=\int_\mathbb{N} f(n+m) dE_{\omega}(n)$ for each $m\geq1$,
  satisfying $E_{\omega}(f)=\lim_{n\rightarrow\infty}f(n)$ if the limit exists.
Let $D$ be a continuous derivation from the Banach
algebra $\mathfrak{B}(\mathcal{S})$ into $\mathcal{L}(\mathcal{S})$. For each $\xi,\eta\in \mathcal{H}$, we define
\begin{align}\label{averaging for $X_0$}
\langle A\xi,\eta\rangle:=\int_\mathbb{N} \langle (X_{0}^{*})^{n}D(X_{0}^{n+1})\xi,\eta \rangle dE_{\omega}(n).
\end{align}
Then $A$ is a bounded linear operator  on $\mathcal{H}$. Moreover, we have
\begin{align*}
\langle A\xi,\eta\rangle&=\int \langle (X_{0}^{*})^{n+1}D(X_{0}^{n+2})\xi,\eta \rangle dE_{\omega}(n)\\
&=\int \langle D(X_{0})\xi,\eta \rangle dE_{\omega}(n)+
\int \langle (X_{0}^{*})^{n+1}D(X_{0}^{n+1})X_0\xi,\eta \rangle dE_{\omega}(n)\\
&=\langle D(X_{0})\xi,\eta \rangle+
\langle X_{0}^{*}AX_0\xi,\eta \rangle,
\end{align*}
and it follows that
\begin{align}\label{*spatial for $X_0$}
D(X_{0})=A-X_{0}^{*}AX_0.
\end{align}
Similarly, we define
\begin{align}\label{averaging for $X_1$}
\langle B\xi,\eta\rangle:=\int_\mathbb{N} \langle (X_{1}^{*})^{n}D(X_{1}^{n+1})\xi,\eta \rangle dE_{\omega}(n),
\end{align}
then we have $B \in B(\mathcal{H})$ and
\begin{align}\label{*spatial for $X_1$}
D(X_{1})=B-X_{1}^{*}BX_1.
\end{align}
\subsection{The case of $X_0$}
\label{subsec_The case of $X_0$}
In this subsection, we will prove the following lemma.
\begin{lemma}\label{lem:spatial for $X_0$}
There exists some $\widehat{A}$ in $\mathcal{L}(\mathcal{S})$ such that
$D(X_0)=\widehat{A}X_0-X_0\widehat{A}$.
\begin{proof}
Let
 \[D(X_{0})=\sum_{X_{0}|X}f(X)X+\sum_{X_{0}\nmid X}f(X)X.\] 
 By the fact that $X_{1}X_{0}=X_{0}X_{2}$, we have
  \[ D(X_{1})X_{0}+X_{1}D(X_{0})=X_{0}D(X_{2})+D(X_{0})X_{2}.\]
  It follows that
  \[X_{1}\sum_{X_{0}\nmid X}f(X)X=\sum_{X_{0}\nmid X}f(X)XX_{2}.\]
Since   $X_{1}X\prec XX_{2}$ whenever
  $X_{0}\nmid X$,  hence $f(X)=0$ in this case. Therefore, 
  \[D(X_0)=\sum_{X_{0}|X}f(X)X=X_{0}L_{f_1}\] 
  for some $L_{f_1}$ in $\mathcal{L}(\mathcal{S})$.
By further computations, we can obtain that 
$D(X_0^n)=X_0^nL_{f_n}$ for some $L_{f_n}$ in $\mathcal{L}(\mathcal{S})$. We define 
\begin{align*}
\langle \widehat{A}\xi,\eta\rangle&:=-\int_\mathbb{N} \langle (X_{0}^{*})^{n+1}D(X_{0}^{n+1})\xi,\eta \rangle dE_{\omega}(n),
\end{align*}
then $\widehat{A}\in B(\mathcal{H})$.
For each $T$ in $\mathcal{R}(\mathcal{S})$, we have
\begin{align*}
\langle T\widehat{A}\xi,\eta\rangle&=-\int_\mathbb{N} \langle (X_{0}^{*})^{n+1}D(X_{0}^{n+1})\xi,T^*\eta \rangle dE_{\omega}(n)\\
&=-\int_\mathbb{N} \langle (X_{0}^{*})^{n+1}D(X_{0}^{n+1})T\xi, \eta \rangle dE_{\omega}(n)\\
&=\langle \widehat{A}T\xi,\eta\rangle.
\end{align*}
It follows that $\widehat{A}\in\mathcal{R}(\mathcal{S})'=\mathcal{L}(\mathcal{S})$.
By equation (\ref{averaging for $X_0$}), we have
\begin{align*}
\langle A\xi,\eta\rangle&=\int_\mathbb{N} \langle (X_{0}^{*})^{n}D(X_{0}^{n+1})\xi,\eta \rangle dE_{\omega}(n)\\
&=\int_\mathbb{N} \langle (X_{0}^{*})^{n+1}D(X_{0}^{n+1})\xi,X_0^*\eta \rangle dE_{\omega}(n)\\
&= \langle -\widehat{A}\xi,X_0^*\eta \rangle,
\end{align*}
  which implies $A=-X_0\widehat{A}$. Then  by
  equation (\ref{*spatial for $X_0$}), we have
$D(X_0)=\widehat{A}X_0-X_0\widehat{A}$.
This completes the proof.
  \end{proof}
\end{lemma}
\subsection{The case of $X_1$}\label{subsec_The case of $X_1$}
The following lemma is the main conclusion of this subsection.
\begin{lemma}\label{spatial for $X_1$}
There exists some $\widehat{B}$ in $\mathcal{L}(\mathcal{S})$ such that  
  $D(X_{1})=\widehat{B}X_{1}-X_{1}\widehat{B}$ .
\end{lemma}
We need Lemmas \ref{lem_p_1_*_derivation_complete} and \ref{lem_p_1_conclusion} to obtain
Lemma \ref{spatial for $X_1$}.
  \begin{lemma}\label{lem_p_1_*_derivation_complete}
  $D(X_{1})=B-X_{1}^{*}BX_{1}$ for some
    $B$ in $\mathcal{L}(\mathcal{S})$.
 \begin{proof} 
 Let $D(X_{1})=L_{f}$ in $\mathcal{L}(\mathcal{S})$, then by the continuity of $D$ we have $f(X_{1}^{n})=0$ for each $n\in\mathbb{N}$.
  For each $m\geq1$, we have 
     $(X_{1}^{*})^{m-1}D(X_{1}^{m})=\sum_{i=0}^{m-1}(X_1^*)^iL_fX_1^{i}$
     and
$ \|(X_{1}^{*})^{m-1}D(X_{1}^{m})\|\leq\|D\|.$
Let $h_{1}$ and $ h_{2}$ be two elements in $\mathcal{S}$. If $h_{2}h_{1}^{-1}= X_{1}^{k}$ for some $k\in \mathbb{N}$, then
 \[\lim_{n,m\rightarrow\infty}\left\langle\bigg((X_{1}^{*})^{m-1}D(X_{1}^{m})-(X_{1}^{*})^{n-1}D(X_{1}^{n})\bigg)\delta_{h_{1}},\delta_{h_{2}}\right\rangle=\lim_{n\rightarrow\infty}\sum_{i=n+1}^{m}f(X_{1}^{i}h_{2}h_{1}^{-1}X_{1}^{-i})=0.\] 
If $h_{2}h_{1}^{-1}\neq X_{1}^{k}$ for any $k\in \mathbb{N}$, then
   $X_{1}^{i}h_{2}h_{1}^{-1}X_{1}^{-i}\notin \mathcal{S}$  when $i$ is sufficiently large.
   Therefore,
  \[\lim_{n,m\rightarrow\infty}\left\langle\bigg((X_{1}^{*})^{m-1}D(X_{1}^{m})-(X_{1}^{*})^{n-1}D(X_{1}^{n})\bigg)\delta_{h_{1}},\delta_{h_{2}}\right\rangle=\lim_{n\rightarrow\infty}\sum_{i=n+1}^{m}f(X_{1}^{i}h_{2}h_{1}^{-1}X_{1}^{-i})=0.\]  
We define 
\[\langle T\delta_{h_1},\delta_{h_2}\rangle:=   \lim_{m\rightarrow\infty}\big\langle(X_{1}^{*})^{m-1}D(X_{1}^{m})\delta_{h_{1}},\delta_{h_{2}}\big\rangle.\]
Then $T$ is the weak-topology limit of  $(X_{1}^{*})^{m-1}D(X_{1}^{m})$ in $B(\mathcal{H})$. We have the following claim. 
    
    \textbf{Claim:} $T=L_{T\delta_{e}}\in\mathcal{L}(\mathcal{S})$. To prove this claim, we distinguish two cases:

    \textbf{Case I:} $h_{2}h_{1}^{-1}\in \mathcal{S}$.
    We have
 \[\langle T\delta_{h_{1}},\delta_{h_{2}}\rangle=\sum_{n=0}^{\infty}f(X_{1}^{n}h_{2}h_{1}^{-1}X_{1}^{-n})=\langle T\delta_{e},\delta_{h_{2}h_{1}^{-1}}\rangle=\langle T\delta_{e}\ast\delta_{h_{1}},\delta_{h_{2}}\rangle.\]
 
 \textbf{Case II:} $h_{2}h_{1}^{-1}\notin S$.
If   $X_{1}^{n}h_{2}h_{1}^{-1}X_{1}^{-n}\notin S$ for any  $n\in\mathbb{N}$, then
  \[\langle T\delta_{h_{1}},\delta_{h_{2}}\rangle=\langle T\delta_{e}\ast\delta_{h_{1}},\delta_{h_{2}}\rangle=0.\]
 On the other hand, there exists a natural number $n$ such that $X_{1}^{n+1}h_{2}h_{1}^{-1}X_{1}^{-n-1}\in \mathcal{S}$ and
   $X_{1}^{i}h_{2}h_{1}^{-1}X_{1}^{-i}\notin \mathcal{S}$ whenever $i\leq n$. Let $X=X_{1}^{n+1}h_{2}h_{1}^{-1}X_{1}^{-n-1}$, then $X_{0}|X$. 
By Lemma \ref{lem:partial_order_S_property} we have that $X_{1}^{-n}XX_{1}^{n}\notin \mathcal{S}$ for any $n\geq1$ and $X_{1}^{n}XX_{1}^{-n}\notin \mathcal{S}$ whenever $n>\mathsf{ind}(X)$.
  Let $m$ be an even number such that $m>>\mathsf{ind}(X)$, then we have
  \begin{align*}
  \sum_{i=1}^{\frac{m}{2}}\left|\langle D(X_1^{m})\delta_e,\delta_{X_1^{m-i}XX_1^{i-1}}\rangle \right|^2=
   \sum_{i=1}^{\frac{m}{2}}\left|\sum_{j=-(i-1)}^{m-i}f(X_1^{j}XX_1^{-j})\right|^2=
   \sum_{i=1}^{\frac{m}{2}}\left|\sum_{j=0}^{\mathsf{ind}(X)}f(X_1^{j}XX_1^{-j})\right|^2.
  \end{align*}
 Note that $X_{1}^{m-i}XX_{1}^{i-1}\neq X_{1}^{m-j}XX_{1}^{j-1}$ whenever $1\leq i< j\leq\frac{m}{2}$, then we have
 \[\sum_{i=1}^{\frac{m}{2}}\left|\sum_{j=0}^{\mathsf{ind}(X)}f(X_1^{j}XX_1^{-j})\right|^2 \leq \bigg\|D(X_{1}^{m})\delta_{e}\bigg\|_{2}^{2} \leq\bigg\|D\bigg\|^{2} \]
 for any $m$. This implies that $\sum_{j=0}^{\mathsf{ind}(X)}f(X_1^{j}XX_1^{-j})=0$.
  Therefore,
   \[ \langle T\delta_{h_{1}},\delta_{h_{2}}\rangle=\sum_{j=0}^{\infty}f(X_{1}^{j}h_{2}h_{1}^{-1}X_{1}^{-j})
  =\sum_{j=0}^{\mathsf{ind}(X)}f(X_{1}^{j}XX_{1}^{-j})=0.\]
  From the above discussion, we have $\langle T\delta_{h_{1}},\delta_{h_{2}}\rangle=\langle T\delta_{e}\ast\delta_{h_{1}},\delta_{h_{2}}\rangle$
for all $h_{1}$ and $h_{2}$ in $\mathcal{S}$. 
 Thus the claim holds.
 
  By equation (\ref{averaging for $X_1$}), we have $B=T$. This completes the proof.
    \end{proof}
  \end{lemma}
  
   \begin{lemma}\label{lem_p_1_conclusion}
 Let $D(X_{1})=L_{f}\in\mathcal{L}(\mathcal{S})$, then
 \[D(X_{1})=\sum_{\ X_{0}|X}f(X)X+\sum_{ X_{0}\nmid X,X_{1}|X}f(X)X.\]
  \begin{proof}
  Let $D(X_{1})=\sum_{X_{0}|X}f(X)X+\sum_{X_{0}\nmid X,X_{1}|X}f(X)X+\sum_{X_{0}\nmid X,X_{1}\nmid X}f(X)X$. By the fact that $X_{1}X_{3}=X_{2}X_{1}$,
  we have
   \[D(X_{1})X_{3}+X_{1}D(X_{3})=D(X_{2})X_{1}+X_{2}D(X_{1}).\]
   Then
   \[\sum_{X_{0}\nmid X,X_{1}\nmid X}f(X)XX_{3}=X_{2}\sum_{X_{0}\nmid X,X_{1}\nmid X}f(X)X.\]
Let $X$ be the minimal element of the set $\{X\in\mathcal{S}:\ X_{0}\nmid X,\ X_{1}\nmid X,\ f(X)\neq0\}$. Then $X_2X\prec XX_3$.
   It follows that $f(X)=0$ when $X_{0}\nmid X$ and $X_{1}\nmid X$.
 This completes the proof.
  \end{proof}
  \end{lemma}
  Now, we are ready to prove Lemma \ref{spatial for $X_1$}.
  \begin{proof}[Proof of Lemma \ref{spatial for $X_1$}]
   By Lemma \ref{lem_p_1_*_derivation_complete}, we have $D(X_{1})=B-X_{1}^{*}BX_{1}$ for some $B=L_{g}\in\mathcal{L}(\mathcal{S})$. We may assume that $g(e)=0$, where $e$ is the unit element of $\mathcal{S}$.
   Let
  \[L_{g}=\sum_{X_{1}|X}g(X)X+\sum_{X_{1}\nmid X,X_{0}|X}g(X)X+\sum_{X_{1}\nmid X,X_{0}\nmid X}g(X)X.\]
  Then
  \begin{align*}
  D(X_{1})=&\sum_{X_{1}|X}g(X)X-X_{1}^{*}\sum_{X_{1}|X}g(X)XX_{1}+\sum_{X_{1}\nmid X,X_{0}\nmid X}g(X)X-X_{1}^{*}\sum_{X_{1}\nmid X,X_{0}\nmid X}g(X)XX_{1}
  \end{align*}
  \begin{align*}
  +&\sum_{X_{1}\nmid X,X_{0}|X}g(X)X-X_{1}^{*}\sum_{X_{1}\nmid X,X_{0}|X}g(X)XX_{1}.
  \end{align*}
  Since $D(X_{1})\in\mathcal{L}(\mathcal{S})$ and $X_1^{-1}XX_1\notin \mathcal{S}$ when $X_{1}\nmid X$ and $X_{0}|X$ , we have $\sum_{X_{1}\nmid X,X_{0}|X}g(X)X=0$. Therefore,
  \[D(X_{1})=\sum_{X_{1}|X}g(X)X-X_{1}^{*}\sum_{X_{1}|X}g(X)XX_{1}+\sum_{X_{1}\nmid X,X_{0}\nmid X}g(X)X-X_{1}^{*}\sum_{X_{1}\nmid X,X_{0}\nmid X}g(X)XX_{1}.\]
  By Lemma \ref{lem_p_1_conclusion}, we have
    \[\sum_{X_{1}\nmid X,X_{0}\nmid X}g(X)X-X_{1}^{*}\sum_{X_{1}\nmid X,X_{0}\nmid X}g(X)XX_{1}=0.\]
    Let $X$ be the minimal element of the set $\{X\in\mathcal{S}:\ X_{1}\nmid X,\ X_{0}\nmid X,\ g(X)\neq0\}$, then $X\prec X_1^{-1}XX_1$ since $X\neq e$. 
  It follows that $g(X)=0$ when  $X_{1}\nmid X$ and 
  $X_{0}\nmid X$. Thus $B=\sum_{X_{1}|X}g(X)X$. Let $\widehat{B}=-X_{1}^{*}B$, then $\widehat{B}\in\mathcal{L}(S)$ and  $D(X_{1})=B-X_{1}^{*}BX_{1}=X_{1}X_{1}^{*}B-X_{1}^{*}BX_{1}=\widehat{B} X_{1}-X_{1}\widehat{B}$. We complete the proof.
  \end{proof}

\subsection{Conditional expectation}
\begin{definition}\label{def_conditional expectation}
Let $\mathfrak{B}$ be a Banach algebra and $\mathcal{A}$ be a Banach subalgebra of $\mathfrak{B}$. Let $E$: $\mathfrak{B}\rightarrow \mathcal{A}$ be a contraction such that 

($\romannumeral1$) $E(I)=I$; 

($\romannumeral2$)  $E(A_1BA_2)=A_1E(B)A_2$ whenever $A_1$, $A_2$ $\in \mathcal{A}$ and $B$ $\in\mathfrak{B}$. 

\noindent Then $E$ is described as a \textbf{conditional expectation} from $\mathfrak{B}$ onto $\mathcal{A}$. 
\end{definition}
\begin{definition}
Let   
 $\mathcal{L}_0(\mathcal{S})$ be the subset of $\mathcal{L}(\mathcal{S})$ such that for each $L_f$
 in $\mathcal{L}_0(\mathcal{S})$, $f(X)=0$ if $X\neq
 X_0^n$.
 \end{definition}
 It is not difficult to check that  $\mathcal{L}_0(\mathcal{S})$ is a Banach subalgebra of $\mathcal{L}(\mathcal{S})$.
We have the following theorem.
\begin{theorem}\label{thm_conditional expectation}
Let $E$ be the map: $\mathcal{L}(\mathcal{S})$
$\rightarrow$ $\mathcal{L}_0(\mathcal{S})$,
$\sum_{X\in\mathcal{S}} f(X)X$ $\rightarrow$ $\sum_{n=0}^{\infty} f(X_0^n)X_0^n$. Then $E$
is a well-defined conditional expectation from $\mathcal{L}(\mathcal{S})$ onto $\mathcal{L}_0(\mathcal{S})$.
\begin{proof}
For each $g\in l^2(\mathcal{S})$ such that
$g(X)=0$ when $X\neq X_0^n$, we have
\begin{align*}
\bigg\|\sum_{n=0}^{\infty} f(X_0^n)X_0^n \sum_{n=0}^{\infty}g(X_0^n)\delta_{X_0^n} \bigg\|_2^2\leq\bigg \|L_fg\bigg\|_2^2\leq \bigg\|L_f\bigg\|^2\bigg\|g\bigg\|_2^2.
\end{align*}
This implies that $\sum_{n=0}^{\infty} f(X_0^n)X_0^n$ is bounded on the Hilbert subspace $l^2(\mathcal{S}_1)$, where $\mathcal{S}_1$ is the subsemigroup of $\mathcal{S}$ generated by $X_0$.
The norm of $\sum_{n=0}^{\infty} f(X_0^n)X_0^n$ is bounded by $\|L_f\|$.
Let $\mathscr{F}$ be the Fourier transform: $\mathbb{Z}\rightarrow \mathbb{S}^1$, $n\rightarrow e^{2\pi i \theta}$, $\theta\in [0,1)$. This induces the following isomorphisms \cite{Ge}:
\begin{align*}
&\mathbb{Z}\quad \subseteq\quad\ \ \mathbb{C}\mathbb{Z}\quad\quad \subseteq \quad l^1(\mathbb{Z})\ \quad \subseteq \quad C^*(\mathbb{Z})\quad \subseteq \quad\  \mathcal{L}(\mathbb{Z})\quad\ \ \subseteq\quad l^2(\mathbb{Z})\quad\ \subseteq\quad \cdots\\
&\updownarrow \quad\quad\quad\quad\ \updownarrow
 \quad\quad\quad\quad\quad\ \updownarrow
  \quad\quad\quad\quad\quad\quad \updownarrow
   \quad\quad\quad\quad\quad\ \updownarrow \quad\quad\quad\quad\quad\ \ \updownarrow
\\
&\mathbb{S}^1\quad \subseteq\quad \mathbb{C}[Z]_{\mathbb{S}^1}\quad \subseteq \quad R(\mathbb{S}^1)\quad \subseteq \quad C(\mathbb{S}^1)\quad \subseteq \quad L^{\infty}(\mathbb{S}^1)\quad \subseteq\quad l^2(\mathbb{S}^1)\quad \subseteq\quad \cdots
\end{align*}
Restricting the Fourier transform on $\mathbb{N}$, we have 
\begin{align*}
&\mathbb{C}\mathbb{N}\quad\quad \subseteq \quad\ l^1(\mathbb{N})\quad\ \subseteq \quad B(\mathbb{N})\quad\ \subseteq \quad\ \ \mathcal{L}(\mathbb{N})\quad\ \subseteq\quad\ l^2(\mathbb{N})\quad\ \subseteq\quad \cdots\\
&\ \updownarrow
 \quad\quad\quad\quad\quad\quad \updownarrow
  \quad\quad\quad\quad\quad\ \updownarrow
   \quad\quad\quad\quad\quad\quad \updownarrow \quad\quad\quad\quad\quad\ \ \updownarrow
\\
&\mathbb{C}[Z]_{\mathbb{D}}\quad \subseteq \quad H^1(\mathbb{D})\quad \subseteq \quad H_c(\mathbb{D})\quad \subseteq \quad H^{\infty}(\mathbb{D})\quad \subseteq\quad H^2(\mathbb{D})\quad \subseteq\quad \cdots
\end{align*}
Since $\mathcal{S}_1$ is isomorphic to $\mathbb{N}$, therefore $\sum_{n=0}^{\infty} f(X_0^n)\delta_{n}$ belongs to $\mathcal{L}(\mathbb{N})$ and
$\mathscr{F}(\sum_{n=0}^{\infty} f(X_0^n)\delta_{n})$ is in $H^{\infty}(\mathbb{D})$.
It follows that $\mathscr{F}(\sum_{n=0}^{\infty} f(X_0^n)\delta_{n})$ belongs to $L^{\infty}(\mathbb{S}^1)$. Thus $\sum_{n=0}^{\infty} f(X_0^n)\delta_{n}$ is a bounded operator on $l^2(\mathbb{Z})$ and belongs to $\mathcal{L}(\mathbb{Z})$. Since the subgroup $\left\langle X_0\right\rangle$ 
of Thompson's group $\mathcal{F}$ is isomorphic to $\mathbb{Z}$, we have that 
$\sum_{n=0}^{\infty} f(X_0^n)X_0^n$ is bounded on $l^2(\mathcal{F})$. Thus $\sum_{n=0}^{\infty} f(X_0^n)X_0^n$ is in $\mathcal{L}_0(\mathcal{S})$. Then $E$ is well-defined. It is clear that conditions ($\romannumeral1$) and ($\romannumeral2$) in Definition \ref{def_conditional expectation}  hold for $E$. We complete the proof.
\end{proof}
\end{theorem}

\subsection{Proof of main results}  
In this subsection, we shall show the main conclusions 
(Theorems \ref{thm_the first cohomology is zero} and \ref{thm_first_co_B(S)})
of this paper. When the $\mathfrak{B}(\mathcal{S})$-bimodule is $\mathcal{L}(\mathcal{S})$,
  the first continuous cohomology group of $\mathfrak{B}(\mathcal{S})$ vanishes. 
   \begin{theorem}\label{thm_the first cohomology is zero}
   $H^1(\mathfrak{B}(\mathcal{S}),\mathcal{L}(\mathcal{S}))=0$.
   \end{theorem}
   The following result is an immediate corollary of
   Theorems \ref{prop: continuity_derivations} and \ref{thm_the first cohomology is zero}.
\begin{corollary}\label{cor:derivation_spatial}
Derivations on the Banach algebra $\mathfrak{B}(\mathcal{S})$ are spatial and induced by linear operators in $\mathcal{L}(\mathcal{S})$.
\end{corollary}
Let $\mathcal{M}(\mathcal{S})$ be the subset of $\mathcal{L}(\mathcal{S})$ such that $[T,A]\in \mathfrak{B}(\mathcal{S})$ for each $T\in \mathcal{M}(\mathcal{S})$ and $A\in \mathfrak{B}(\mathcal{S})$. Then $\mathcal{M}(\mathcal{S})$ is a norm closed linear subspace of $\mathcal{L}(\mathcal{S})$ containing $\mathfrak{B}(\mathcal{S})$. Then the first cohomology group of $\mathfrak{B}(\mathcal{S})$
is charactered as the linear space $\mathcal{M}(\mathcal{S})$ module $\mathfrak{B}(\mathcal{S})$.
\begin{theorem}\label{thm_first_co_B(S)}
$H^{1}(\mathfrak{B}(\mathcal{S}),\mathfrak{B}(\mathcal{S}))=\dfrac{\mathcal{M}(\mathcal{S})}{\mathfrak{B}(\mathcal{S})}$.
\end{theorem}

The following two lemmas are crucial to obtain 
the above results.
\begin{lemma}\label{lem_spacial_derivation_1}
  Let $D$ be a continuous derivation from the Banach algebra $\mathfrak{B}(\mathcal{S})$ into
  $\mathcal{L}(\mathcal{S})$. If 
   $D(X_0)=AX_{0}-X_{0}A$ and $D(X_{1})=AX_{1}-X_{1}A$ for some 
   $A$ in $B(l^{2}(\mathcal{S}))$, then $D(T)=AT-TA$ for each $T$ in $\mathfrak{B}(\mathcal{S})$.
  \begin{proof}
For each $n\geq1$,  if $D(X_{0}^{n})=AX_{0}^{n}-X_{0}^{n}A$, then
 \[D(X_{0}^{n+1})=X_{0}D(X_{0}^{n})+D(X_{0})X_{0}^{n}=AX_{0}^{n+1}-X_{0}^{n+1}A.\]
 Therefore, by induction we have  $D(X_{0}^{n})=AX_{0}^{n}-X_{0}^{n}A$ for any $n\geq1$.
    By the definition of Thompson's semigroup, we have $X_{1}X_{0}^{m}=X_{0}^{m}X_{m+1}$ for any
    $m\geq1$. Then
     \[D(X_{1})X_{0}^{m}+X_{1}D(X_{0}^{m})=D(X_{0}^{m})X_{m+1}+X_{0}^{m}D(X_{m+1}).\]
     By computation, we have $D(X_{m+1})=AX_{m+1}-X_{m+1}A$. Similarly, it can be proved that
     $D(X)=AX-XA$  for any $X$ in $\mathcal{S}$. 
By linearity, we have $D(T)=AT-TA$ for each $T$
 in the semigroup algebra    $\mathbb{C}\mathcal{S}$. Since $\mathbb{C}\mathcal{S}$ is a dense subalgebra of $\mathcal{B}(\mathcal{S})$, we obtain that  $D(T)=AT-TA$ for each $T$ in $\mathfrak{B}(\mathcal{S})$ by the continuity of $D$.
  \end{proof}
  \end{lemma}
  The following lemma is a generalization of the above
  conclusion.
   \begin{lemma}\label{lem_spacital_derivation_2}
 Let $D$ be a continuous derivation from the Banach algebra $\mathfrak{B}(\mathcal{S})$ into
  $\mathcal{L}(\mathcal{S})$. If 
   $D(X_0)=AX_{0}-X_{0}A$ and $D(X_{1})=BX_{1}-X_{1}B$ for some 
   $A$ and $B$ in $\mathcal{L}(\mathcal{S})$, then
there exists an operator $C$ in $\mathcal{L}(\mathcal{S})$ such that  
    $D(T)=CT-TC$ for each $T$ in $\mathfrak{B}(\mathcal{S})$.   
  \begin{proof}
  By the definitions of Thompson's semigroup and derivations, we have
  \[X_{1}X_{0}=X_{0}X_{2},\ X_{2}X_{0}=X_{0}X_{3},\ X_{2}X_{1}=X_{1}X_{3}\]
  and
  \begin{align*}
  D(X_{1})X_{0}+X_{1}D(X_{0})&=D(X_{0})X_{2}+X_{0}D(X_{2}),\\
   D(X_{2})X_{0}+X_{2}D(X_{0})&=D(X_{0})X_{3}+X_{0}D(X_{3}),\\
   D(X_{2})X_{1}+X_{2}D(X_{1})&=D(X_{1})X_{3}+X_{1}D(X_{3}).
  \end{align*}
By the above equations, we have
  \begin{gather*}
  D(X_{2})=X_{0}^{*}X_{1}(A-B)X_{0}+X_{0}^{*}(B-A)X_{1}X_{0}+(AX_{2}-X_{2}A),\\
  D(X_{3})=(X_{0}^{*})^{2}X_{1}(A-B)X_{0}^{2}+(X_{0}^{*})^{2}(B-A)X_{1}X_{0}^{2}+(AX_{3}-X_{3}A),\\
 D(X_{3})=X_{1}^{*}X_{0}^{*}X_{1}(A-B)X_{0}X_{1}+X_{1}^{*}X_{0}^{*}(B-A)X_{1}X_{0}X_{1}
 +X_{1}^{*}(A-B)X_{2}X_{1}\\
+X_{1}^{*}X_{2}(B-A)X_{1}+(BX_{3}-X_{3}B).
 \end{gather*}
 It follows that
 \begin{gather*}
(X_{0}^{*})^{2}X_{1}(A-B)X_{0}^{2}+(X_{0}^{*})^{2}(B-A)X_{1}X_{0}^{2}+(A-B)X_{3}+X_{3}(B-A)=\\
 X_{1}^{*}X_{0}^{*}X_{1}(A-B)X_{0}X_{1}+X_{1}^{*}X_{0}^{*}(B-A)X_{1}X_{0}X_{1}+X_{1}^{*}(A-B)X_{2}X_{1}+X_{1}^{*}X_{2}(B-A)X_{1}.
 \end{gather*}
 Let $A-B=L_{f}\in \mathcal{L}(\mathcal{S})$ and  $L_{g}=L_{f}-f(e)I-\sum_{n\geq0}\sum_{m\geq1}f(X_{n}^{m})X_{n}^{m}$. 
 Then we have
   \begin{equation} \label{equation1 of lem 5.5}
   \begin{split}
   (X_{0}^{*})^{2}X_{1}L_{g}X_{0}^{2}-(X_{0}^{*})^{2}L_{g}X_{1}X_{0}^{2}+L_{g}X_{3} - X_{3}L_{g}=X_{1}^{*}X_{0}^{*}X_{1}L_{g}X_{0}X_{1}-X_{1}^{*}X_{0}^{*}L_{g}X_{1}X_{0}X_{1}\\
   +X_{1}^{*}L_{g}X_{2}X_{1}-X_{1}^{*}X_{2}L_{g}X_{1}+\left(\sum_{n\geq2}\sum_{m\geq1} f(X_{n}^{m})X_{n+1}^{m}\right)X_{3}- X_{3}\left(\sum_{n\geq2}\sum_{m\geq1}f(X_{n}^{m})X_{n+1}^{m}\right)\ \\
\quad\quad+X_{3}\left(\sum_{n\geq2}\sum_{m\geq1}f(X_{n}^{m})X_{n}^{m}\right)-\left(\sum_{n\geq2}\sum_{m\geq1}f(X_{n}^{m})X_{n}^{m}\right)X_{3}\quad\quad\quad\quad\quad\quad\quad.
   \end{split}
   \end{equation}
If $g\neq0$, then take the minimal element $X$ of the set $\{X\in \mathcal{S}\ |\ g(X)\neq0\}$ under the total order.
 By the definition of $L_{g}$, we have 
$X=X_{i_{1}}^{\alpha_{i_{1}}}\cdots X_{i_{t}}^{\alpha_{i_{t}}}$, where $i_{1}<i_{2}<\cdots<i_{t},\ t\geq2$, and $\alpha_{i_{1}},\cdots,\alpha_{i_{t}}\geq1$.

\textbf{Case I:} $i_{1}\geq3$. We have $X_{3}X\prec XX_{3}$.
Taking $\langle\cdot\ \delta_{e},\delta_{X_{3}X}\rangle$ at the both sides of equation (\ref{equation1 of lem 5.5}), we obtain that $g(X)=0$. 

\textbf{Case II:}  $i_{1}=1$ or $2$. We have $ XX_{3}\prec X_{3}X$. Analogously, taking $\langle\cdot\ \delta_{e},\delta_{XX_{3}}\rangle$ at the both sides of equation (\ref{equation1 of lem 5.5}), we also have $g(X)=0$.

\textbf{Case III:} $i_{1}=0$. If $X_1|XX_{2}X_{1}$, then taking
$\langle\cdot\ \delta_{e},\delta_{X_{1}^{-1}XX_{2}X_{1}}\rangle$, we can obtain that $g(X)=0$. If $X_1\nmid XX_{2}X_{1}$, by the normal form of Thompson's group $\mathcal{F}$ there exist $Y$ and $X_{k}$ in $\mathcal{S}$ such that $X_{1}^{-1}XX_{2}X_{1}=YX_{k}^{-1}$, where $X_{1}^{-1}$ and $X_{k}^{-1}$ are in $\mathcal{F}$.
Then taking $\langle\cdot\ \delta_{X_{k}},\delta_{Y}\rangle$, we can also obtain that $g(X)=0$.

It follows from the above discussion that $g=0$. This leads to 
\begin{align*}
 \left(\sum_{n\geq2}\sum_{m\geq1}f(X_{n}^{m})X_{n+1}^{m}\right)X_{3}-&X_{3}\left(\sum_{n\geq2}\sum_{m\geq1}f(X_{n}^{m})X_{n+1}^{m}\right)\\
 +&X_{3}\left(\sum_{n\geq2}\sum_{m\geq1}f(X_{n}^{m})X_{n}^{m}\right)-\left(\sum_{n\geq2}\sum_{m\geq1}f(X_{n}^{m})X_{n}^{m}\right)X_{3}\\
 =&0.
\end{align*}
 It is not difficult to verify that $f(X_2^m)=0$ 
and $f(X_n^m)=f(X_3^m)$ for all $n \geq 4$ and $m\geq1$. Since $f\in l^2(\mathcal{S})$, we have
  $f(X_{n}^{m})=0$ for any $n\geq2$ and $m\geq1$. Therefore, $A-B=\sum_{m=1}^{\infty}f(X_{0}^{m})X_{0}^{m}+\sum_{m=1}^{\infty}f(X_{1}^{m})X_{1}^{m}$+$f(e)I$.
 Let 
 \[C=A-\sum_{m=1}^{\infty}f(X_{0}^{m})X_{0}^{m}=B+\sum_{m=1}^{\infty}f(X_{1}^{m})X_{1}^{m}+f(e)I.\] 
  Then by Theorem \ref{thm_conditional expectation} we have  $C\in\mathcal{L}(\mathcal{S})$  and
 \[D(X_{0})=CX_{0}-X_{0}C,\quad D(X_{1})=CX_{1}-X_{1}C.\]
 By Lemma \ref{lem_spacial_derivation_1}, we have $D(T)=CT-TC$ for each $T$ in $\mathfrak{B}(\mathcal{S})$.   
  \end{proof}
  \end{lemma}

 \begin{proof}[Proof of Theorem \ref{thm_the first cohomology is zero}]
 Let $D$ be a continuous derivation from $\mathfrak{B}(\mathcal{S})$ into $\mathcal{L}(\mathcal{S})$.
  By  Lemmas \ref{lem:spatial for $X_0$},
\ref{spatial for $X_1$} and
   \ref{lem_spacital_derivation_2}, there exits a $C$ in $\mathcal{L}(S)$ such that $D(T)=CT-TC$ for each $T$ in $\mathfrak{B}(\mathcal{S})$. Thus $D$ is an inner
   derivation and  $H^1(\mathfrak{B}(\mathcal{S}),\mathcal{L}(\mathcal{S}))=0$. We complete the proof.
 \end{proof}

  \begin{proof}[Proof of Theorem 
  \ref{thm_first_co_B(S)}]
  Let $D$ be a derivation on the Banach algebra
  $\mathfrak{B}(\mathcal{S})$, then $D$ is continuous.
   By Lemmas \ref{lem:spatial for $X_0$},
\ref{spatial for $X_1$} and
   \ref{lem_spacital_derivation_2}, there exists an operator $C$ in $\mathcal{L}(\mathcal{S})$ such that
   $D(T)=CT-TC$ for each $T$ in $\mathfrak{B}(\mathcal{S})$. 
  This induces the following map:
 \begin{align*}
\pi:\quad H^{1}(\mathfrak{B}(\mathcal{S}),\mathfrak{B}(\mathcal{S}))\quad&\longrightarrow\qquad \frac{ \mathcal{M}(\mathcal{S})}{\mathfrak{B}(\mathcal{S})}\\
\overline{D}\qquad\qquad&\longrightarrow\quad\ C+\mathfrak{B}(\mathcal{S}).
 \end{align*}
 We shall show that $\pi$ is well-defined. If there exist two operators $C_1$ and $C_2$ in $\mathcal{L}(\mathcal{S})$ such that $D(T)=C_1T-TC_1=C_2T-TC_2$ for each $T$ in $\mathfrak{B}(\mathcal{S})$, then $C_1-C_2\in \mathcal{L}(\mathcal{S}) \cap \mathfrak{B}(\mathcal{S})'=\mathcal{L}(\mathcal{S}) \cap \mathcal{R}(\mathcal{S})=\mathbb{C}I$. Thus $C_1+\mathfrak{B}(\mathcal{S})=C_2+\mathfrak{B}(\mathcal{S})$. Now if $\overline{D_1}=\overline{D_2}$, then $D_1-D_2$ is an inner derivation of $\mathfrak{B}(\mathcal{S})$. There exists an operator $C_3$ in $\mathfrak{B}(\mathcal{S})$ such that
   $(D_1-D_2)(T)=C_3T-TC_3$ for each $T$ in $\mathfrak{B}(\mathcal{S})$. 
  Assume that $D_1(T)=C_1'T-TC_1'$ and $D_2(T)=C_2'T-TC_2'$, where $C_1'$ and $C_2'$
  are in $\mathcal{L}(\mathcal{S})$,
   then $(C_1'-C_2')T-T(C_1'-C_2')=C_3T-TC_3$. It follows that $C_1'-C_2'-C_3$
  belongs to  $\mathcal{L}(\mathcal{S}) \cap \mathfrak{B}(\mathcal{S})'=\mathbb{C}I$.
Thus  $C_1'-C_2'$
  belongs to  $ \mathfrak{B}(\mathcal{S})$, that is
  $C_1'+\mathfrak{B}(\mathcal{S})=C_2'+\mathfrak{B}(\mathcal{S})$. The map $\pi$ is a well-defined and is a group homomorphism.
 If $\pi(\overline{D})=0$, then there exists an operator $C'$ in $\mathfrak{B}(\mathcal{S})$ such that $D(T)=C'T-TC'$, which means that $D$ is an inner derivation.
The map $\pi$ is injective. The surjectivity of $\pi$ is obvious. It follows from the above discussion that $\pi$ is a group isomorphism. We complete the proof.
  \end{proof}
    \section{Further discussions}\label{sec_Further discussions}
     We now recall the definition of higher order continuous  Hochschild cohomology for Banach algebras.
     Let $\mathcal{M}$ be a Banach algebra and $\mathcal{X}$ be a Banach $\mathcal{M}$-bimodule. 
     The space of all $n$-linear
(continuous) maps from $n$-fold Cartesian product  $\mathcal{M}^n=\mathcal{M}\times\cdots\times\mathcal{M}$ into $\mathcal{X}$
is denoted by  $L^{n}(\mathcal{M},\mathcal{X})$ for $n\geq1$,
while $L^{0}(\mathcal{M},\mathcal{X})$ is defined to be $\mathcal{X}$.

The coboundary operator $\partial^n$ : $L^{n}(\mathcal{M},\mathcal{X})\rightarrow L^{n+1}(\mathcal{M},\mathcal{X})$ is defined,  
for $n\geq1$, by
\begin{align*}
\partial^n\phi(a_{1},a_{2},\ldots,a_{n+1})=&a_{1}\phi(a_{2},\ldots,a_{n+1})\\
+&\sum_{i=1}^{n}(-1)^{i}\phi(a_1,\ldots,a_{i-1},a_ia_{i+1},a_{i+2},\ldots,a_{n+1})\\
+&(-1)^{n+1}\phi(a_{1},\ldots,a_{n})a_{n+1}
\end{align*}
where $\phi\in L^{n}(\mathcal{M},\mathcal{X})$ and $a_1,\ a_2,\ldots,\ a_{n+1}\in\mathcal{M}$. 
When $n=0$, we define $\partial^0$ by
\[\partial^0 x(m)=mx-xm\quad\quad\quad (x\in\mathcal{X},\ m\in\mathcal{M}).\]
It is routine to check that
 $\partial^{n}\partial^{n-1}$ : $L^{n-1}(\mathcal{M},\mathcal{X})$
$\rightarrow$ $L^{n+1}(\mathcal{M},\mathcal{X})$
is zero for all $n\geq1$, and so $Im(\partial^{n-1}:\ L^{n-1}(\mathcal{M},\mathcal{X})\rightarrow L^{n}(\mathcal{M},\mathcal{X}))$ is a linear subspace of
$Ker(\partial^n:\ L^{n}(\mathcal{M},\mathcal{X})\rightarrow L^{n+1}(\mathcal{M},\mathcal{X}))$.
 The $n^{th}$ Hochschild cohomology group $H^{n}(\mathcal{M},\mathcal{X})$ is then
defined to be the following quotient space
\[\frac{Ker(\partial^n:\ L^{n}(\mathcal{M},\mathcal{X})\rightarrow L^{n+1}(\mathcal{M},\mathcal{X}))}{Im(\partial^{n-1}:\ L^{n-1}(\mathcal{M},\mathcal{X})\rightarrow L^{n}(\mathcal{M},\mathcal{X}))}\]
   for $n\geq1$. We end this paper by proposing some problems for future study:
\begin{itemize}
\item  What are the higher order cohomology groups 
$H^{n}(\mathfrak{B}(S),\mathfrak{B}(S))$ for $n\geq2$?
\item  When $n\geq2$, does  $H^n(\mathfrak{B}(\mathcal{S}),\mathcal{L}(\mathcal{S}))=0$? 
The first step to calculate the high order cohomology groups should be the following. Given a 2-cocyle 
$\phi$,
we need to modify it by a 1-coboundary such that
$\phi$ is $X_0$-multimodular, i.e., $\phi(X_0A,B)= X_0\phi(A,B)$, $\phi(AX_0,B)= \phi(A,X_0B)$, and $\phi(A,BX_0)= \phi(A,B)X_0$ for all $A,B\in \mathfrak{B}(\mathcal{S})$.
\end{itemize}  


\newpage
\bibliographystyle{plain}

\begin{thebibliography}{99}
{
\bibitem{BJ} W. G. Bade and P. C. Curtis, Jr., The continuity of derivations of Banach algebras. 
{\sl J. Funct. Anal.} \textbf{16} (1974), 372--387. 

\bibitem{BS} M.G. Brin and C.C. Squier, Groups of piecewise linear homeomorphisms of the real line, { \sl Invent. Math.}
\textbf{79} (1985), 485--498.


\bibitem{CFP}
  J. W. Cannon, W. J. Floyd and W. R. Parry,
    Introductory notes on Richard Thompson's groups,
  {\sl Enseign. Math. (2)} \textbf{42} (1996), 215--256.

 \bibitem{CPSS} E. Christensen, F. Pop, A. M. Sinclair 
and R. R. Smith, Hochschild cohomology of factors
with property $\Gamma$, {\sl Ann. of Math. (2)}
\textbf{158} (2003), 635--659.

\bibitem{DHX} A. Dong, L. Huang and B. Xue,  Operator algebras associated with multiplicative
convolutions of arithmetic functions, {\sl Sci. China Math.} \textbf{61} (2018), 1665--1676.

\bibitem{Folner}
E. F{\o}lner, On groups with full Banach mean value,
{\sl Math. Scand.} \textbf{3} (1955), 243--254.

\bibitem{Frey}
 A. H. Frey, Studies on amenable semigroups. PhD thesis, {\sl University of Washington} 1960.

\bibitem{Ge} L. Ge, Numbers and figures (in Chinese), in The   
Institute Lecture 2010, Nanhua Xi edited, {\sl Science Press} 2012, 1--10.







\bibitem{Ho1} G. Hochschild,
   On the cohomology groups of an associative algebra,
   {\sl Ann. of Math. (2)}
   \textbf{46} (1945), 58--67.
    
\bibitem{Ho2} G. Hochschild,
     On the cohomology theory for associative algebras,
   {\sl Ann. of Math. (2)}
  \textbf{47} (1946), 568--579.
  
  \bibitem{Ho3} G. Hochschild,
     Cohomology and representations of associative algebras, {\sl Duke Math. J.}
    \textbf{14} (1947), 921--948.
     
\bibitem{Joh} B. E. Johnson, Cohomology in Banach Algebras, {\sl Memoirs of the American Mathematical
Society. American Mathematical Society}, 1972.

     \bibitem {J} B. E. Johnson,
     A class of {${\rm II}\sb{1}$} factors without property {${\rm
              P}$} but with zero second cohomology,
   {\sl Ark. Mat.}
  \textbf{12} (1974), 153--159.
     
      \bibitem {JKR}
    B. E. Johnson, R. V. Kadison and J. R. Ringrose,
    Cohomology of operator algebras III. Reduction to normal cohomology,
    {Bull. Soc. Math. France}
  \textbf{100} (1972), 73--96. 
   
    \bibitem{JS}B. E. Johnson and A. M. Sinclair, Continuity of derivations and a problem of Kaplansky,
 {\sl Amer. J. Math.} \textbf{90} (1968), 1067--1973.
   
    \bibitem {Kad2} R. V. Kadison,
    On the orthogonalization of operator representations,
   {\sl Amer. J. Math.}
  \textbf{77} (1955), 600--620.
   
   \bibitem{Kad} R. V. Kadison, Derivations of operator algebras, {\sl Ann. of Math.} \textbf{83} (1966), 280--293.
     
      \bibitem {KR} R. V. Kadison and J. R. Ringrose,
Cohomology of operator algebras I: Type I von Neumann algebras,
    {\sl  Acta Math.}
  \textbf{126} (1971), 227--243.
 

\bibitem{KR2} R. V. Kadison and J. R. Ringrose,
     Cohomology of operator algebras II: Extended cobounding and the hyperfinite case,
    {\sl Ark. Mat.}
   \textbf{9} (1971), 55--63.
     
     
\bibitem{Kap}
I. Kaplansky, Modules over operator algebras, {\sl Amer. J. Math.} \textbf{75} (1953), 839--859.

\bibitem{PS} F. Pop and R. R. Smith, Vanishing of second cohomology for tensor products
of type II$_1$ von Neumann algebras, {\sl J. Funct. Anal.} \textbf{258} (2010), 2695--2707.

\bibitem{Sa1}
S. Sakai, On a conjecture of Kaplansky, {\sl Tohoku Math. J. (2)} \textbf{12} (1960), 31--33.

\bibitem{Sa} S. Sakai,
  Derivations of {W}$^{*}$-Algebras,
  {\sl Ann. of Math. (2)} \textbf{83} (1966), 273--279.
 


\bibitem{SS2} A. M. Sinclair and  R. R. Smith,
Hochschild cohomology of von Neumann algebras,
{\sl Cambridge University Press}, 1995. 
 
 
\bibitem{SS} A. M. Sinclair and R. R. Smith, Hochschild cohomology for von Neumann algebras with Cartan subalgebras, {\sl Amer. J. Math.} \textbf{120} (1998), 1043--1057.


\bibitem{Xue}B. Xue,  Natural monoids and non-commutative arithmetics, {\sl arXiv:1901.02149}, 2019. 

 }

\end{thebibliography}

\end{document}